\documentclass[11pt,reqno]{amsart}

\usepackage{amsmath,amssymb,mathrsfs}
\usepackage{graphicx,cite}
\usepackage{dsfont}
\usepackage{color}
\usepackage{inputenc}
\usepackage[T1]{fontenc}

\newtheorem{theo}{Theorem}[section]

\newtheorem{lemm}[theo]{Lemma}
\newtheorem{prop}[theo]{Proposition}
\newtheorem{rema}[theo]{Remark}
\newtheorem{defi}[theo]{Definition}

\newtheorem{assu}{A}
\renewcommand{\Re}{\operatorname{Re}}
\renewcommand{\Im}{\operatorname{Im}}
\numberwithin{equation}{section}

\begin{document}

\title[Subwavelength resonances in high contrast elastic media]{
Analysis of subwavelength resonances in high contrast elastic media by a variational method}

\author{Bochao Chen}
\address{School of Mathematics and Statistics, Center for Mathematics and Interdisciplinary Sciences, Northeast Normal University, Changchun, Jilin 130024, China}
\email{chenbc758@nenu.edu.cn}

\author{Yixian Gao}
\address{School of Mathematics and Statistics, Center for Mathematics and Interdisciplinary Sciences, Northeast Normal University, Changchun, Jilin 130024, China}
\email{gaoyx643@nenu.edu.cn}

\author{Peijun Li}
\address{LSEC, ICMSEC, Academy of Mathematics and Systems Science, Chinese Academy of Sciences, Beijing 100190, China, and School of Mathematical Sciences, University of Chinese Academy of Sciences, Beijing 100049, China}
\email{lipeijun@lsec.cc.ac.cn}

\author{Yuanchun Ren}
\address{School of Mathematics and Statistics, Center for Mathematics and Interdisciplinary Sciences, Northeast Normal University, Changchun, Jilin 130024, China}
\email{renyuanchun@nenu.edu.cn}

\thanks{The research of BC was supported in part by NSFC grant (project number 12471181) and the Fundamental Research Funds for the Central Universities (project numbers, 2412022QD032, 2412023YQ003 and GFPY202303). The research of YG was supported by NSFC grants (project numbers, 12371187 and 12071065) and Science and Technology Development Plan Project of Jilin Province 20240101006JJ}



\begin{abstract}
In this paper, we present a mathematical study of wave scattering by a hard elastic obstacle embedded in a soft elastic body in three dimensions. Our contributions are threefold. First, we characterize subwavelength resonances using the Dirichlet-to-Neumann map and an auxiliary variational form, showing that these resonances occur when the determinant of a specific matrix vanishes. Second, employing Gohberg--Sigal theory and Puiseux series expansions for multi-valued functions, we derive the asymptotic expansions of subwavelength resonant frequencies in the low-frequency regime through this explicit characterization. Finally, we provide a representation of the scattered field in the interior domain, where the enhancement coefficients are governed by the imaginary parts of the resonant frequencies. Additionally, we establish the transversal and longitudinal far-field patterns for the scattered field in the exterior domain.
\end{abstract}

\maketitle

\section{Introduction}\label{sec:1}

Metamaterials are engineered materials with unique properties that enable the manipulation and control of waves, resulting in remarkable effects such as super-resolution imaging \cite{CLHGMS2010}, invisibility cloaking \cite{KLSW2014,LLL2018}, waveguides \cite{AHY2021}, and superfocusing \cite{LPLFLT2015,AFGLZ2017}. High contrast resonators, which can generate subwavelength resonances, are often ideal components for constructing metamaterials due to their low dissipation loss, in contrast to nanoparticles with negative parameters, which are another type of metamaterial. The term ``subwavelength'' refers to obstacles that are significantly smaller than the wavelength of the incident wave. Resonators are devices designed to amplify acoustic, electromagnetic, or elastic waves at specific frequencies. The study of subwavelength resonances caused by high contrast has received considerable attention in recent years, though most research has focused on acoustic and electromagnetic scattering problems.

In acoustics, the resonators of interest are Minnaert bubbles \cite{AFGLZ2017,Fepponmodal,Ammari2018}, which have densities much lower than that of water. However, due to the instability of bubbles in liquid, soft elastic materials with small shear moduli, such as silicone rubber, are often used as the background medium for metamaterial design \cite{Calvo2015underwater}. For such coupled acoustic and elastic wave systems, a mathematical theory of subwavelength resonances has been developed in \cite{Li2022minnaert}. In electromagnetics, high contrast dielectric particles are commonly used to achieve subwavelength resonances \cite{AAMYB2016,Libowenmathematical}.

The study of elastic wave scattering poses significant challenges due to the coupling of transverse and longitudinal waves. It has been mathematically shown that polariton resonances caused by negative material parameters can be described as eigenvalue problems of the Neumann--Poincar\'{e} operator \cite{LLL2018,BochaoE}. In contrast, high contrast elastic materials can generate local resonances and band gaps, leading to wave attenuation, primarily observed through physical and numerical simulations \cite{Ding2007Metamaterial,sheng2003locally,LZMZYCS2000}. Lead balls coated with silicone rubber are frequently employed as resonators to induce subwavelength resonances, with dipole resonances shown to achieve negative mass density \cite{LZMZYCS2000,Liu2005Analytic}. This demonstrates that, while the mechanisms behind subwavelength resonances differ from those of polarization resonances, the effective medium theory provides a conceptual bridge between the two phenomena.

Despite the evidence presented in the physical literature, there remains a lack of comprehensive mathematical understanding of elastic scattering and subwavelength resonances caused by high contrast factors. To our knowledge, only the work by \cite{LZ2024} provides mathematical support for \cite{LZMZYCS2000,Liu2005Analytic,Ding2007Metamaterial}, which focus on the fabrication of sonic crystals with a hard solid core coated in soft elastic material. The approach in \cite{LZ2024} utilizes layer potential techniques and is closely related to the spectra of boundary integral operators. Additionally, we would like to highlight the reference \cite{FA2024}, which examines acoustic wave scattering induced by a Minnaert bubble through a variational approach. This method demonstrates considerable flexibility and has been extended to accommodate scenarios involving a finite number of resonators \cite{FCA2023} as well as infinite periodic chains of resonators \cite{ABCF2023} in acoustics.

Building on the insights provided by \cite{FA2024}, we analyze the scattering of a low-frequency incident wave by a hard elastic obstacle embedded in a soft elastic material from a variational perspective. We reformulate the elastic scattering problem in $\mathbb{R}^3$ as a boundary value problem defined on the bounded domain $D$ by introducing the Dirichlet-to-Neumann (DtN) map. Utilizing an auxiliary sesquilinear form, we characterize the well-posedness of the corresponding variational problem in terms of the determinant of a matrix. Unlike the well-established Poincar\'{e}--Wirtinger inequality applicable to the Helmholtz equation, the analysis of elastic waves is complicated by the intricate nature of the Lam\'{e} operator. Consequently, we must establish a new inequality to derive a sesquilinear form that is both bounded and coercive.

For the elastic scattering problem, the well-posedness analysis draws upon the Fredholm alternative as demonstrated in \cite{Li2019Inverse,Bao2021adaptive,DtN:Li} and employs the Lax--Milgram theorem from \cite{Gao2018Analysis} within a variational framework. Additionally, the space spanned by all linear solutions to the scalar equation:
\begin{align*}
\left\{ \begin{aligned}
&\Delta{u}=0&&\text{in } D,\\
&\nabla{u}\cdot\boldsymbol{\nu}=0&&\text{on } \partial D
\end{aligned}\right.
\end{align*}
is one-dimensional, enabling relevant calculations to be performed in the scalar case. Here $\boldsymbol{\nu}$ is the outward unit normal to the boundary $\partial D$. Conversely, the space spanned by solutions, satisfying the following vector system:
\begin{align}\label{traction 0}
\left\{ \begin{aligned}
&\mathcal{L}_{\lambda,\mu}\boldsymbol{u}=0&&\text{in } D,\\
&\partial_{\boldsymbol{\nu} }\boldsymbol{u}=0&&\text{on } \partial D
\end{aligned}\right.
\end{align}
with related symbols described in Section \ref{sec:2}, is six-dimensional. Thus, the resonant frequencies correspond to the eigenvalues of a matrix. The authors in \cite{Fepponmodal, FCA2023, ABCF2023} have employed the implicit function theorem to derive the asymptotic expansions of frequencies at resonance, contingent on the assumption that the eigenvalues of the correlation matrix are simple. Our analysis of resonance is grounded in Gohberg--Sigal theory \cite{Lay}, which includes the generalized argument principle and Rouch\'{e}'s theorem. Given that the eigenvalues of a matrix may exhibit multiplicities, it becomes essential to derive the expansion of Puiseux series for multi-valued functions.

The main contributions of this paper are threefold. First, to study the existence and uniqueness of the weak solution to the interior scattering problem, we introduce an auxiliary sesquilinear form that is both bounded and coercive. This is achieved by adding a suitable term:
\begin{equation*}
\sum_{i=1}^6\int_{D}\boldsymbol{u}\cdot\boldsymbol{\xi}_i\mathrm{d}\boldsymbol{x}\int_{D}
\overline{\boldsymbol{v}}\cdot\boldsymbol{\xi}_i\mathrm{d}\boldsymbol{x},
\end{equation*}
where $(\boldsymbol{\xi}_i)_{1\leq i\leq 6}$ is a set of canonical orthogonal basis vectors for the solution space of equation \eqref{traction 0}. In addition, the auxiliary variational form plays a crucial role in characterizing the resonance frequencies, specifically through the condition that the determinant of a matrix vanishes. In the absence of resonances, the solution can be expressed in the following decomposition form:
\begin{align*}
\boldsymbol{u}(\omega,\delta)=\sum^{6}_{j=1}s_{j}(\omega,\delta)\boldsymbol{w}_j(\omega,\delta)
+\boldsymbol{w}_{\boldsymbol f}(\omega,\delta),
\end{align*}
where $\boldsymbol{w}_j(\omega,\delta)$ and $\boldsymbol{w}_{\boldsymbol f}(\omega,\delta)$ are holomorphic functions in $\omega$ and $\delta$ and can be uniquely determined according to Lax--Milgram theorem. It is important to note that $\boldsymbol{w}_j$ is independent of the incident field and behaves as a small perturbation of the field $\boldsymbol{\xi}_j$?. When the incident wave is at any resonant frequency, the scattering of the field is primarily reflected in the resonance amplitudes $(s_{j})_{1\leq j\leq 6}$?. Second, we explicitly calculate the asymptotic expansions of the resonance frequencies, whose leading terms are associated with the eigenvalues of a real-valued symmetric positive definite matrix. These expansions clearly demonstrate that subwavelength behavior is determined by the high contrast in densities, even though the high contrast in the Lam\'{e} parameters is the primary cause of the resonances. Furthermore, the resonance frequencies exhibit non-positive imaginary principal terms and are symmetric with respect to the imaginary axis. Finally, using asymptotic analysis, we show that the expansion of the field in $D$ includes frequencies appearing in the denominator, indicating that the field will diverge when the incident frequency approaches any calculated resonant frequency. It is noteworthy that the imaginary components of the resonant frequencies play a crucial role in field enhancement within the resonator. In addition, we include the far-field pattern, which is divided into two components: the transversal and longitudinal far-field patterns.

The organizational structure of the paper is divided into four main sections. Section \ref{sec:2} presents the mathematical formulation of the elastic system, along with the definitions and properties of the layer potential operators and the DtN map. In Section \ref{sec:3}, we provide asymptotic expressions for the resonant frequencies based on an explicit characterization of subwavelength resonances. Section \ref{sec:4} focuses on the asymptotic expansion  and field enhancement in the domain $D$, as well as the far-field expansion of the scattering solution. Finally, Section \ref{sec:conclusion} concludes the paper with some remarks and suggestions for future research directions.

\section{Problem formulation and preliminaries}\label{sec:2}

In this section, we present the problem formulation of the elastic system, along with the definitions and key properties of the layer potential operators and the DtN map.

\subsection{Problem formulation}

Let $D\subset \mathbb{R}^3$ be a bounded and simply connected  domain with a boundary $\partial D$ of class $C^{2}$.
Define a fourth-rank isotropic elastic tensor $\widehat{\mathds{C}}=(\widehat{C}_{ijkl})^3_{i,j,k,l=1}$ by
\begin{align*}
\widehat{C}_{ijkl}:=\widehat{\lambda}\delta_{ij}\delta_{kl}+\widehat{\mu}(\delta_{ik}\delta_{jl}+\delta_{il}\delta_{jk}),
\end{align*}
where $\delta$ is the Kronecker delta, and $(\widehat\lambda,\widehat{\mu})$ are the
Lam\'{e} parameters, defined as
\begin{align*}
(\mathbb{R}^3; \widehat\lambda, \widehat\mu)=(D; \widetilde{\lambda}, \widetilde{\mu})\cup (\mathbb{R}^3\backslash\overline{D}; \lambda, \mu),
\end{align*}
which describe the configuration of the elastic media. Moreover, we assume that the Lam\'{e} parameters $(\lambda,\mu)$ of the background medium in $\mathbb{R}^3\backslash\overline{D}$ satisfy the following strong convexity conditions:
\begin{align}\label{convexity}
\mu>0,\quad 3\lambda+2\mu>0.
\end{align}

Let $\boldsymbol{u}:=\boldsymbol{u}(\boldsymbol{x})=(u_i(\boldsymbol{x}))_{i=1}^3$ be the elastic displacement field. Define the Lam\'{e} operator $\mathcal{L}_{\widehat{\lambda},\widehat{\mu}}$ by
\begin{align*}
\mathcal{L}_{\widehat{\lambda},\widehat{\mu}}\boldsymbol{u}:=\nabla\cdot\widehat{\mathds{C}}\big(\frac{1}{2}(\nabla \boldsymbol{u}+\nabla\boldsymbol{u}^{\top})\big)=\widehat{\mu}\nabla\cdot(\nabla\boldsymbol{u}+\nabla\boldsymbol{u}^\top)+\widehat{\lambda}\nabla\nabla\cdot\boldsymbol{u},
\end{align*}
and define the conormal derivative on $\partial D $ by
\begin{align*}
\partial_{\boldsymbol{\nu}}\boldsymbol{u}:=\widehat{\lambda}(\nabla \cdot \boldsymbol{u})\boldsymbol{\nu}+\widehat{\mu}(\nabla \boldsymbol{u}+\nabla \boldsymbol{u}^\top)\boldsymbol{\nu},
\end{align*}
where $\nabla \boldsymbol{u}$ stands for the matrix $(\partial_{x_j}u_i)_{i,j=1}^3$, $\nabla \boldsymbol{u}^\top$ represents the transpose of $\nabla \boldsymbol{u}$, and $\boldsymbol{\nu}$ is the outward unit normal to the boundary $\partial D$. Furthermore, we note that the traces on $\partial D$ taken from the exterior and interior of $D$ are denoted by $+$ and $-$, respectively.

The constants $\rho$ and $\widetilde{\rho}$ represent the mass densities of the background medium in $\mathbb{R}^3\setminus\overline{D}$ and the medium in the region $D$, respectively. Let $c_s$, $c_p$, $\widetilde{c}_s$, and $\widetilde{c}_p$ denote the shear and compressional wave velocities in $\mathbb{R}^3\setminus\overline{D}$ and $D$, respectively, defined by
\begin{align}\label{velocity}
c_s=\sqrt{\mu/\rho}, \quad  c_p=\sqrt{(\lambda+2\mu)/\rho},\quad\widetilde{c}_s=\sqrt{\widetilde{\mu}/\widetilde{\rho}},\quad \widetilde{c}_p=\sqrt{(\widetilde{\lambda}+2\widetilde{\mu})/\widetilde{\rho}},
\end{align}
where $(\mu, \lambda)$ are the Lam\'{e} parameters in the background medium, and $(\widetilde{\mu}, \widetilde{\lambda})$ are the corresponding parameters in the region $D$. Let
\begin{align}\label{wave number}
k_s=\frac{\omega}{c_s},\quad k_p=\frac{\omega}{c_p}
\end{align}
represent the shear and compressional wavenumbers, respectively, in the background medium $\mathbb{R}^3\setminus\overline{D}$.

Let $\boldsymbol{u}^{\rm in}$ denote the incident field satisfying
\begin{align*}
\mathcal{L}_{\lambda,\mu}\boldsymbol{u}^{\text{in}}+\rho \omega^2\boldsymbol{u}^{\text{in}}=0 \quad\text{in } \mathbb{R}^3.
\end{align*}
The displacement of the total field $\boldsymbol u$ satisfies the time-harmonic elastic wave scattering problem, which is characterized by the Lam\'{e} system:
\begin{align}\label{Lame2}
\left\{ \begin{aligned}
&\mathcal{L}_{\lambda,\mu}\boldsymbol{u}+\rho\omega^2\boldsymbol{u}=0&&\text{in }\mathbb{R}^3\backslash\overline{D},\\
&\mathcal{L}_{\widetilde{\lambda},\widetilde{\mu}}\boldsymbol{u}+\widetilde{\rho}\omega^2\boldsymbol{u}=0&&\text{in } D,\\
&\boldsymbol{u}|_+=\boldsymbol{u}|_-&&\text{on } \partial D,\\
&\partial_{\boldsymbol{\nu}} \boldsymbol{u}\big|_+=\partial _{\boldsymbol{\nu}}\boldsymbol{u}\big|_-&&\text{on }\partial D.
\end{aligned}\right.
\end{align}
 The third and fourth equations in \eqref{Lame2} represent the transmission conditions. As described in
\cite{BLZ-JMPA20}, the displacement of the scattered field, $\boldsymbol{u}-\boldsymbol{u}^{\mathrm{in}}$, can be expressed as the sum of a compressional component $\boldsymbol u_p$ and a shear component $\boldsymbol u_s$, i.e.,
\[
 \boldsymbol{u}-\boldsymbol{u}^{\mathrm{in}}=\boldsymbol u_p+ \boldsymbol u_s\quad\text{in } \mathbb R^3\setminus\overline D,
\]
where $\boldsymbol u_p$ and $\boldsymbol u_s$ satisfy the Kupradze--Sommerfeld radiation condition:
\begin{equation}\label{src}
 \partial_r \boldsymbol u_p -{\rm i}k_p\boldsymbol u_p=\mathcal{O}(|\boldsymbol x|^{-2}),\quad \partial_r \boldsymbol u_s -{\rm i}k_s\boldsymbol u_s=\mathcal{O}(|\boldsymbol x|^{-2})
\end{equation}
as $|\boldsymbol x|\to\infty$, which hold uniformly in all directions $\widehat{\boldsymbol x}:=\boldsymbol x/|\boldsymbol x|\in\mathbb{S}^{2}$, where $\mathbb{S}^2$ is the unit sphere in $\mathbb{R}^3$. These conditions ensure that the scattered wave is outgoing and that the solution to the scattering problem \eqref{Lame2}--\eqref{src} is unique (cf. \cite[Theorem 3.7]{Li2019Inverse}).

Consider the contrast in the Lam\'{e} constants and mass densities between hard and soft elastic materials. Specifically, these contrasts are given by the relationships:
\begin{align}\label{ratio}
(\widetilde{\lambda},\widetilde{\mu})=\frac{1}{\delta}(\lambda,\mu), \quad\widetilde{\rho}=\frac{1}{\epsilon}\rho.
\end{align}
Additionally, we introduce a parameter $\tau$ to represent the contrast in shear and compressional wave velocities between different regions:
\begin{align}\label{contrast}
\tau=\frac{c_s}{\widetilde{c}_s}=\frac{c_p}{\widetilde{c}_p}=\sqrt{\frac{\delta}{\epsilon}},
\end{align}
where the last equality follows from \eqref{velocity} and \eqref{ratio}. Furthermore, we assume that
\begin{align*}
\tau=\mathcal{O}(1),
\end{align*}
which implies that the contrast in wave velocities is of order one.

Using \eqref{ratio} and \eqref{contrast}, the system in \eqref{Lame2} can be reformulated as follows:
\begin{align}\label{Lame}
\left\{ \begin{aligned}
&\mathcal{L}_{\lambda,\mu}\boldsymbol{u}+\rho\omega^2\boldsymbol{u}=0&&\text{in }\mathbb{R}^3\backslash\overline{D},\\
&\mathcal{L}_{\lambda,\mu}\boldsymbol{u}+\rho\tau^2\omega^2\boldsymbol{u}=0&&\text{in } D,\\
&\boldsymbol{u}|_+=\boldsymbol{u}|_-&&\text{on } \partial D,\\
&\delta\partial_{ \boldsymbol{\nu}} \boldsymbol{u}\big|_+=\partial_{\boldsymbol{\nu}}\boldsymbol{u}\big|_-&&\text{on }\partial D
\end{aligned}\right.
\end{align}
together with the Kupradze--Sommerfeld radiation condition \eqref{src}. In this work, we focus on the scattering problem in \eqref{Lame} and \eqref{src} under the subwavelength and high-contrast limits:
\begin{align*}
\omega\rightarrow 0,\quad\delta\rightarrow 0^+,\quad\epsilon\rightarrow0^+.
\end{align*}

\subsection{Layer potential operators}

Let $\mathbf{I}$ be the $3\times3$ the identity matrix. For $k\neq0$, the fundamental solution to the elastic wave operator $\mathcal{L}_{\lambda,\mu}+k^2$ is given by the Kupradze matrix $\boldsymbol\Gamma^{k}=(\Gamma^{k}_{ij})^{3}_{i,j=1}$, expressed as
\begin{align*}
\boldsymbol{\Gamma}^{k}(\boldsymbol x)
=-\frac{e^{\mathrm{i}\frac{k}{\sqrt{\mu}}|\boldsymbol x|}}{4\pi\mu|\boldsymbol x|}\mathbf{I}+\frac{1}{4\pi k^2}\nabla\nabla\left(\frac{e^{\mathrm{i}\frac{k}{\sqrt{\lambda+2\mu}}|\boldsymbol x|}
-e^{\mathrm{i}\frac{k}{\sqrt{\mu}}|\boldsymbol x|}}{|\boldsymbol x|}\right).
\end{align*}
In the specific case of $k=0$, the Kupradze matrix reduces the Kelvin matrix $\boldsymbol\Gamma^0=(\Gamma_{ij})^{3}_{i,j=1}$, which is written as
\begin{align*}
\boldsymbol{\Gamma}^{0}(\boldsymbol x)=-\frac{1}{8\pi}\left(\frac{1}{\mu}+\frac{1}{\lambda+2\mu}\right)\frac{1}{|\boldsymbol x|}\mathbf{I}-\frac{1}{8\pi}\left(\frac{1}{\mu}-\frac{1}{\lambda+2\mu}\right)\frac{\boldsymbol x {\boldsymbol x}^{\top}}{|\boldsymbol x|^3}.
\end{align*}

For $k\in\mathbb{R}, \boldsymbol\varphi\in H^{-\frac{1}{2}}(\partial D)^3$, and $\boldsymbol x\in\mathbb{R}^3\backslash\partial D$, the single-layer potential corresponding to $\boldsymbol\Gamma^{k}$ is denoted by
\begin{align*}
\boldsymbol{\widetilde{\mathcal{S}}}^{k}_{D}[\boldsymbol\varphi](\boldsymbol x)=\int_{\partial D}\boldsymbol\Gamma^{k}(\boldsymbol x-\boldsymbol y)\boldsymbol\varphi(\boldsymbol y)\mathrm{d}\sigma(\boldsymbol y).
\end{align*}
In addition, for $\boldsymbol x\in\partial{D}$, we define the single-layer potential and the Neumann--Poincar\'{e} operators as follows:
\begin{align*}
\boldsymbol{\mathcal{S}}^{k}_{D}[\boldsymbol\varphi](\boldsymbol x)&=\int_{\partial D}\boldsymbol\Gamma^{k}(\boldsymbol x-\boldsymbol y)\boldsymbol\varphi(\boldsymbol y)\mathrm{d}\sigma(\boldsymbol y),\\
\boldsymbol{\mathcal{K}}^{k,*}_{D}[\boldsymbol\varphi](\boldsymbol x)&=\mathrm{p.v.}\int_{\partial D} \partial_{\boldsymbol{\nu}_{\boldsymbol x}}\boldsymbol\Gamma^{k}(\boldsymbol x-\boldsymbol y)\boldsymbol\varphi(\boldsymbol y)\mathrm{d}\sigma(\boldsymbol y),
\end{align*}
where the notation $\mathrm{p.v.}$ indicates the Cauchy principal value of the integral. Furthermore, the following jump relations hold for the conormal derivative of the single-layer potential:
\begin{align}\label{Jump2}
\partial_{\boldsymbol{\nu}}\boldsymbol{\mathcal{S}}^{k}_{D}[\boldsymbol\varphi]\big|_{\pm}(\boldsymbol x)=\big(\pm\frac{1}{2}\boldsymbol{\mathcal{I}}+\boldsymbol{\mathcal{K}}^{k,*}_{ D}\big)[\boldsymbol\varphi](\boldsymbol x),\quad \boldsymbol x\in\partial{D}.
\end{align}

In the following,  the elastostatic operators $\boldsymbol{\widetilde{\mathcal{S}}}^{0}_{D}$, $\boldsymbol{\mathcal{S}}^{0}_{D}$, and $\boldsymbol{\mathcal{K}}^{0,*}_{D}$ are denoted by $\boldsymbol{\widetilde{\mathcal{S}}}_{D}$, $\boldsymbol{\mathcal{S}}_{D}$, and $\boldsymbol{\mathcal{K}}^{*}_{D}$, respectively, for simplicity of notation.

\begin{rema}
Given the smooth dependence of $\boldsymbol\Gamma^{k}$ on $k\in\mathbb{C}$, the boundary integral operators $\boldsymbol{\mathcal{S}}^{k}_{D}$ and $\boldsymbol{\mathcal{K}}^{k,*}_{D}$? can be analytically extended to the complex plane.
\end{rema}

The following lemma provides the series expansion of the fundamental solution $\boldsymbol{\Gamma}^{k}$ via a Taylor series expansion with respect to $k$ (cf. \cite[p.28]{Lay}).

\begin{lemm}\label{le:series}
As $k\rightarrow 0$, the fundamental solution $\boldsymbol{\Gamma}^{k}$ admits the following series expansion:
\begin{align*}
\boldsymbol{\Gamma}^{k}(\boldsymbol x)&=-\frac{1}{4\pi}\sum^{+\infty}_{j=0}\frac{\mathrm{i}^j}{(j+2)j!}\left(\frac{j+1}{\mu^{\frac{j+2}{2}}} +\frac{1}{(\lambda+2\mu)^{\frac{j+2}{2}}}\right)k^j|\boldsymbol x|^{j-1}\mathbf{I}\\
&\quad+\frac{1}{4\pi}\sum^{+\infty}_{j=0}\frac{\mathrm{i}^j(j-1)}{(j+2)j!}\left(\frac{1}{\mu^{\frac{j+2}{2}}}
-\frac{1}{(\lambda+2\mu)^{\frac{j+2}{2}}}\right)k^j|\boldsymbol x|^{j-3}\boldsymbol{x}\boldsymbol{x}^{\top}.
\end{align*}
\end{lemm}

Based on Lemma \ref{le:series}, we can derive the asymptotic expansions for the single-layer potential and the  Neumann--Poincar\'{e} operators (cf. \cite{chang2007spectral,BochaoE,Irina1999spectral}).

\begin{lemm}
The single-layer potential operator $\boldsymbol{\mathcal{S}}^{k}_{D}:H^{-\frac{1}{2}}(\partial D)^3\rightarrow H^{\frac{1}{2}}(\partial D)^3$ admits the following asymptotic expansion as $k\rightarrow 0$:
\begin{align*}
\boldsymbol{\mathcal{S}}^{k}_{D}=\boldsymbol{\mathcal{S}}_{D}+k \boldsymbol{\mathcal{S}}_{D,1}+\mathcal{O}(k^2),
\end{align*}
where, for all $\boldsymbol\varphi\in H^{-\frac{1}{2}}(\partial D)^3$,
\begin{align}\label{ri}
\boldsymbol{\mathcal{S}}_{D,1}[\boldsymbol\varphi](\boldsymbol x)&=-\mathrm{i}\gamma\int_{\partial D}\boldsymbol\varphi(\boldsymbol y)\mathrm{d}\sigma(\boldsymbol y), \quad \gamma=\frac{1}{12\pi}\left(\frac{2}{\mu^{\frac{3}{2}}}+\frac{1}{(\lambda+2\mu)^{\frac{3}{2}}}\right).
\end{align}
\end{lemm}

\begin{lemm}
The Neumann--Poincar\'{e} operator $\boldsymbol{\mathcal{K}}^{k,*}_{D}:H^{-\frac{1}{2}}(\partial D)^3\rightarrow H^{-\frac{1}{2}}(\partial D)^3$ has the asymptotic expansion as $k\rightarrow 0$:
\begin{align}\label{EK-series}
\boldsymbol{\mathcal{K}}^{k,*}_{D}=\boldsymbol{\mathcal{K}}^*_{D}+\mathcal{O}(k^2).
\end{align}
\end{lemm}

\begin{lemm}
As $k\rightarrow0$, the boundary integral operator $\boldsymbol{\mathcal{S}}^{k}_{D}$ is invertible, and its inverse is denoted by $(\boldsymbol{\mathcal{S}}^{k}_{D})^{-1}$. The operator $(\boldsymbol{\mathcal{S}}^{k}_{D})^{-1}: H^{\frac{1}{2}}(\partial D)^3\rightarrow H^{-\frac{1}{2}}(\partial D)^3$ has the  following  asymptotic expansion as $k\rightarrow 0$:
\begin{align}\label{EIS:series}
&(\boldsymbol{\mathcal{S}}^{k}_{D})^{-1}=\boldsymbol{\mathcal{S}}^{-1}_{D}-k\boldsymbol{\mathcal{S}}_{D}^{-1}\boldsymbol{\mathcal{S}}_{D,1}\boldsymbol{\mathcal{S}}_{D}^{-1}+\mathcal{O}(k^2).
\end{align}
\end{lemm}

\subsection{Dirichlet-to-Neumann map}

The purpose of this subsection is to provide the definition and asymptotic expansion of the DtN operator, which maps Dirichlet values on the boundary to Neumann values on the boundary (cf. \cite{DtN:Gunter,DtN:Stephanie}). In particular, for spherical or circular regions, the corresponding DtN map is discussed in \cite{Bao2021adaptive,DtN:Li}.

\begin{defi}\label{DtN:def}
For $k\in \mathbb{R}$, we define the DtN map characterized by the operator $\boldsymbol{\mathcal{M}}^{k}:H^{\frac{1}{2}}(\partial D)^{3}\rightarrow H^{-\frac{1}{2}}(\partial D)^{3}$ as follows:
\begin{align*}
\boldsymbol{\mathcal{M}}^{k}[\boldsymbol{f}]:=\partial_{\boldsymbol{\nu}}\boldsymbol{h}_{\boldsymbol{f}}|_+,
\end{align*}
where $\boldsymbol{h}_{\boldsymbol{f}}$ satisfies the system:
\begin{align*}
\left\{ \begin{aligned}
&\mathcal{L}_{\lambda,\mu}\boldsymbol{h}_{\boldsymbol{f}}+k^{2}\boldsymbol{h}_{\boldsymbol{f}}=0&&\text{in }\mathbb{R}^3\backslash\overline{D},\\
&\boldsymbol{h}_{\boldsymbol{f}}=\boldsymbol{f}&&\text{on }\partial D,\\
&(\nabla\times\nabla\times\boldsymbol{h}_{\boldsymbol{f}})(\boldsymbol x)\times\widehat{\boldsymbol x}-\mathrm{i}k\mu^{-\frac{1}{2}}\nabla\times\boldsymbol{h}_{\boldsymbol{f}}(\boldsymbol x)=\mathcal{O}(|\boldsymbol x|^{-2})&& \text{as }|\boldsymbol x|\rightarrow+\infty,\\
&\widehat{\boldsymbol x}\cdot (\nabla(\nabla\cdot\boldsymbol{h}_{\boldsymbol{f}}))(\boldsymbol x)-\mathrm{i}k(\lambda+2\mu)^{-\frac{1}{2}}\nabla\cdot\boldsymbol{h}_{\boldsymbol{f}}(\boldsymbol x)=\mathcal{O}(|\boldsymbol x|^{-2})&&\text{as }|\boldsymbol x|\rightarrow+\infty.
\end{aligned}\right.
\end{align*}
Here, the last two equations represent the Kupradze--Sommerfeld radiation condition. In particular, we denote $\boldsymbol{\mathcal{M}}^0$ simply as $\boldsymbol{\mathcal{M}}$.
\end{defi}

In the following, we set $k=\sqrt{\rho}\omega$ with $\omega\rightarrow 0$.  Consider the Neumann boundary value problem:
\begin{align}\label{Lame:traction 0}
\left\{ \begin{aligned}
&\mathcal{L}_{\lambda,\mu}\boldsymbol{u}=0&&\text{in } D,\\
&\partial_{\boldsymbol{\nu} }\boldsymbol{u}=0&&\text{on } \partial D.
\end{aligned}\right.
\end{align}
This system admits six linearly independent solutions (cf. \cite{ando2018spectral}), which are
\begin{align*}\left[\begin{array}{l}
1 \\
0 \\
0
\end{array}\right],\left[\begin{array}{l}
0 \\
1 \\
0
\end{array}\right],\left[\begin{array}{l}
0 \\
0 \\
1
\end{array}\right],\left[\begin{array}{c}
x_2 \\
-x_1 \\
0
\end{array}\right],\left[\begin{array}{c}
x_3 \\
0 \\
-x_1
\end{array}\right],\left[\begin{array}{c}
0 \\
x_3 \\
-x_2
\end{array}\right],
\end{align*}
where, for $i=1,2,3$, $x_i$ is the $i$-th component of $\boldsymbol{x}\in D$. Denote by $\Xi$ the vector space consisting of all linear solutions to \eqref{Lame:traction 0}. Then the space $\Xi$ can be represented as (cf. \cite[p.14]{Ammari:book2015}):
\begin{align*}
\Xi=\left\{\boldsymbol{u}:\mathcal{E}_{ij}(\boldsymbol{u})=0,i,j=1,2,3\right\},
\end{align*}
where the strain tensor $\mathcal{E}_{ij}(\boldsymbol{u})$ is defined as
\begin{align}\label{E:Eij}
\mathcal{E}_{ij}(\boldsymbol{u}):=\frac{1}{2}(\partial_{x_i}u_j+\partial_{x_j}u_i).
\end{align}

Using the Gram--Schmidt orthogonalization method, we can obtain an  orthonormal basis $\{\boldsymbol\xi_i\}_{1\leq i\leq6}$ for the space $\Xi$ in the $L^{2}(D)^3$ sense, which is given by
\begin{align*}&\boldsymbol\xi_1:=\frac{1}{|D|}\left[\begin{array}{l}
1 \\
0 \\
0
\end{array}\right]=\frac{1}{|D|}\mathbf{e}_1,\quad\boldsymbol\xi_2:=\frac{1}{|D|}\left[\begin{array}{l}
0 \\
1 \\
0
\end{array}\right]=\frac{1}{|D|}\mathbf{e}_2,\quad\boldsymbol\xi_3:=\frac{1}{|D|}\left[\begin{array}{l}
0 \\
0 \\
1
\end{array}\right]=\frac{1}{|D|}\mathbf{e}_3,\\
&\boldsymbol\xi_4:=d_4\left[\begin{array}{c}
x_2-c_2 \\
-x_1+c_1 \\
0
\end{array}\right],\quad\boldsymbol\xi_5:=d_5\left[\begin{array}{c}
x_3-c_3-L_1(x_2-c_2)\\
L_1(x_1-c_1) \\
-x_1+c_1
\end{array}\right],\\
&\boldsymbol\xi_6:=d_6\left[\begin{array}{c}
-L_2(x_2-c_2)-L_3(x_3-c_3)+L_3L_1(x_2-c_2)\\
x_3-c_3+L_2(x_1-c_1)+L_3L_1(c_1-x_1) \\
-x_2+c_2+L_3(x_1-c_1)
\end{array}\right].
\end{align*}
Here, the constants $d_i$ are chosen such that $(\boldsymbol\xi_i,\boldsymbol\xi_i)_{L^2(D)^3}=1,i=4,5,6.$ The vectors $\{\mathbf{e}_n\}_{n=1,2,3}$ form the standard orthonormal basis of $\mathbb{R}^3$. The constants $c_1, c_2, c_3$? represent the centroid of the domain $D$ and are defined as:
\begin{align*}
c_1=\frac{\int_Dx_1\mathrm{d}\boldsymbol{x}}{|D|},\quad c_2=\frac{\int_Dx_2\mathrm{d}\boldsymbol{x}}{|D|},\quad c_3=\frac{\int_Dx_3\mathrm{d}\boldsymbol{x}}{|D|}.
\end{align*}
The coefficients $L_1, L_2, L_3$ are given by
\begin{align*}
 L_1&=\frac{\int_Dx_2x_3\mathrm{d}\boldsymbol{x}-|D|c_2c_3}{\int_D((x_2-c_2)^2+(x_1-c_1)^2)\mathrm{d}\boldsymbol{x}},\quad L_2=\frac{-\int_Dx_1x_3\mathrm{d}\boldsymbol{x}-|D|c_1c_3}{\int_D((x_2-c_2)^2+(x_1-c_1)^2)\mathrm{d}\boldsymbol{x}},\\ L_3&=\frac{\int_D(-L_1x_3(x_2-c_2)-x_2(-x_1+c_1))\mathrm{d}\boldsymbol{x}}{\int_D((x_3-c_3-L_1x_2+L_1c_2)^2+(L_1x_1-L_1c_1)^2+(c_1-x_1)^2)\mathrm{d}\boldsymbol{x}}.
\end{align*}

\begin{lemm}\label{le:xi}
Assume that for $\boldsymbol{\varphi}\in H^{-\frac{1}{2}}(\partial D)^3$, the equation $\boldsymbol{\mathcal{S}}_{D}[\boldsymbol{\varphi}]=\boldsymbol\xi_i $ holds on $\partial D$. Then, for $i=1,\cdots, 6$, the following equation is satisfied:
\begin{align}\label{eq:xi}
\big(-\frac{1}{2}\boldsymbol{\mathcal{I}}+\boldsymbol{\mathcal{K}}^{*}_{ D}\big)\boldsymbol{\mathcal{S}}^{-1}_{D}[\boldsymbol\xi_i]=0\quad \text{on } \partial D.
\end{align}
\end{lemm}

\begin{proof}
Consider the Dirichlet boundary value problem:
\begin{align*}
\left\{ \begin{aligned}
&\mathcal{L}_{\lambda,\mu}\boldsymbol{u}=0&&\text{in }D,\\
&\boldsymbol{u}=\boldsymbol\xi_i&&\text{on }\partial D.
\end{aligned}\right.
 \end{align*}
This system has a unique solution $\boldsymbol{u}=\boldsymbol\xi_i=\widetilde{\boldsymbol{\mathcal{S}}}_{D}[\boldsymbol{\varphi}]$ in $D$. By applying the jump relation \eqref{Jump2} for $k=0$, we obtain
\begin{align*}
\big(-\frac{1}{2}\boldsymbol{\mathcal{I}}+\boldsymbol{\mathcal{K}}^{*}_{ D}\big)[\boldsymbol{\varphi}]=\partial_{\boldsymbol\nu}\boldsymbol\xi_i=0\quad\text{on }{\partial D}.
\end{align*}
Since  $\boldsymbol{\mathcal{S}}_{D}:H^{-\frac{1}{2}}(\partial D)^3\rightarrow H^{\frac{1}{2}}(\partial D)^3$ is invertible, we conclude that $\boldsymbol{\varphi}=\boldsymbol{\mathcal{S}}^{-1}_{D}[\boldsymbol\xi_i]$ on $\partial D$, which implies \eqref{eq:xi} and completes the proof.
\end{proof}

It is clear from Lemma \ref{le:xi} that for $n=1,2,3$,
\begin{align}\label{eq:ei}
\big(-\frac{1}{2}\boldsymbol{\mathcal{I}}+\boldsymbol{\mathcal{K}}^{*}_{ D}\big)\boldsymbol{\mathcal{S}}^{-1}_{D}[\mathbf{e}_n]=0\quad \text{on } \partial D.
\end{align}

Our next objective is to establish the analyticity and asymptotic expansion of the DtN map with respect to the wavenumber $k$ in a neighborhood of zero.

\begin{prop}\label{DtNasy}
Let $k\in\mathbb{C}$. The DtN operator $\boldsymbol{\mathcal{M}}^{k}:H^{\frac{1}{2}}(\partial D)^{3}\rightarrow H^{-\frac{1}{2}}(\partial D)^{3}$ is analytic with respect to $k$ and admits the following asymptotic
expansion as $k\rightarrow 0$:
\begin{align*}
\boldsymbol{\mathcal{M}}^{k}=\boldsymbol{\mathcal{M}}+k\boldsymbol{\mathcal{M}}_1+\mathcal{O}(k^2),
\end{align*}
where, for any $\boldsymbol{f}\in H^{\frac{1}{2}}(\partial D)^{3}$,
\begin{align}
\boldsymbol{\mathcal{M}}[\boldsymbol{f}]&:=\big(\frac{1}{2}\boldsymbol{\mathcal{I}}+\boldsymbol{\mathcal{K}}^{*}_{D}\big)\boldsymbol{\mathcal{S}}^{-1}_{D}[\boldsymbol{f}],\label{eq:M0}\\
\boldsymbol{\mathcal{M}}_1[\boldsymbol f]&:=\mathrm{i}\gamma
\sum_{n=1}^{3}\left(\int_{\partial D}\boldsymbol{\mathcal{M}}[\mathbf{e}_n]\cdot \boldsymbol{f}\mathrm{d}\sigma(\boldsymbol{x})\right)\boldsymbol{\mathcal{M}}[\mathbf{e}_n],\label{eq:M1}
\end{align}
with $\gamma$ being defined by \eqref{ri}.
\end{prop}

\begin{proof}
For $\boldsymbol{\varphi}\in H^{-\frac{1}{2}}(\partial D)^{3}$ and $k\in\mathbb{R}$, if $\boldsymbol{f}=\boldsymbol{\mathcal{S}}^{k}_{D}[\boldsymbol{\varphi}]$ on $\partial D$, then $\boldsymbol{h}_{\boldsymbol{f}}=\widetilde{\boldsymbol{\mathcal{S}}}^{k}_{D}[\boldsymbol{\varphi}]$ in $\mathbb{R}^3\setminus \overline{D}$. It is clear that $\boldsymbol{\varphi}=(\boldsymbol{\mathcal{S}}^{k}_{D})^{-1}[\boldsymbol{f}]$ on $\partial D$. Moreover,
\begin{align}\label{DtN-NP}
\boldsymbol{\mathcal{M}}^{k}[\boldsymbol f]
=\big(\frac{1}{2}\boldsymbol{\mathcal{I}}+\boldsymbol{\mathcal{K}}^{k,*}_{D}\big)[\boldsymbol\varphi]
=\big(\frac{1}{2}\boldsymbol{\mathcal{I}}+\boldsymbol{\mathcal{K}}^{k,*}_{D}\big)(\boldsymbol{\mathcal{S}}^{k}_{D})^{-1}[\boldsymbol f]\quad\text{on } \partial D.
\end{align}
In particular, $\boldsymbol{\mathcal{M}}=\big(\frac{1}{2}\boldsymbol{\mathcal{I}}+\boldsymbol{\mathcal{K}}^{*}_{D}\big)\boldsymbol{\mathcal{S}}^{-1}_{D}$. The analyticity of $\boldsymbol{\mathcal{M}}^{k}$ with respect to $k\in\mathbb{C}$
follows from the analyticity of $\boldsymbol{\mathcal{S}}^{k}_{D}$ and $\boldsymbol{\mathcal{K}}^{k,*}_{D}$ with respect to $k\in\mathbb{C}$.

Combining \eqref{EK-series} with \eqref{EIS:series} yields
\begin{align*}
\boldsymbol{\mathcal{M}}^{k}=\big(\frac{1}{2}\boldsymbol{\mathcal{I}}+\boldsymbol{\mathcal{K}}^{*}_{D}\big)\boldsymbol{\mathcal{S}}^{-1}_{D}
-k\big(\frac{1}{2}\boldsymbol{\mathcal{I}}+\boldsymbol{\mathcal{K}}^{*}_{D}\big)\boldsymbol{\mathcal{S}}^{-1}_{D}\boldsymbol{\mathcal{S}}_{D,1}\boldsymbol{\mathcal{S}}^{-1}_{D}+\mathcal{O}(k^2).
\end{align*}
It follows from \eqref{eq:ei} that
\begin{align*}
\big(\frac{1}{2}\boldsymbol{\mathcal{I}}+\boldsymbol{\mathcal{K}}^{*}_{D}\big)\boldsymbol{\mathcal{S}}^{-1}_{ D}[\mathbf{e}_n]=\boldsymbol{\mathcal{S}}^{-1}_{D}[\mathbf{e}_n],\quad n=1,2,3,
\end{align*}
which gives
\begin{align*}
\boldsymbol{\mathcal{M}}[\mathbf{e}_n]=\big(\frac{1}{2}\boldsymbol{\mathcal{I}}+\boldsymbol{\mathcal{K}}^{*}_{D}\big)\boldsymbol{\mathcal{S}}^{-1}_{ D}[\mathbf{e}_n]=\boldsymbol{\mathcal{S}}^{-1}_{D}[\mathbf{e}_n],\quad n=1,2,3.
\end{align*}
Since $\boldsymbol{\mathcal{S}}^{-1}_{D}$ is self-adjoint, we can deduce that for any $\boldsymbol{f}\in H^{\frac{1}{2}}(\partial D)^{3}$,
\begin{align*}
-\big(\frac{1}{2}\boldsymbol{\mathcal{I}}
+\boldsymbol{\mathcal{K}}^{*}_{D}\big)\boldsymbol{\mathcal{S}}^{-1}_{D}\boldsymbol{\mathcal{S}}_{D,1}\boldsymbol{\mathcal{S}}^{-1}_{D}[\boldsymbol f]&=\mathrm{i}\gamma \big(\frac{1}{2}\boldsymbol{\mathcal{I}}+\boldsymbol{\mathcal{K}}^{*}_{D}\big)
\boldsymbol{\mathcal{S}}^{-1}_{D}\int_{\partial D}\boldsymbol{\mathcal{S}}^{-1}_{D}[\boldsymbol{f}]\mathrm{d}\sigma(\boldsymbol x)\\
&=\mathrm{i}\gamma \big(\frac{1}{2}\boldsymbol{\mathcal{I}}+\boldsymbol{\mathcal{K}}^{*}_{D}\big)
\boldsymbol{\mathcal{S}}^{-1}_{D}\Bigg[\sum^3_{n=1}\bigg(\int_{\partial D}\boldsymbol{\mathcal{S}}^{-1}_{D}[\boldsymbol{f}]\cdot \mathbf{e}_n\mathrm{d}\sigma(\boldsymbol x)\bigg)\mathbf{e}_n\Bigg]\\
&=\mathrm{i}\gamma \big(\frac{1}{2}\boldsymbol{\mathcal{I}}+\boldsymbol{\mathcal{K}}^{*}_{D}\big)
\boldsymbol{\mathcal{S}}^{-1}_{D}\Bigg[\sum^3_{n=1}\bigg(\int_{\partial D}\boldsymbol{\mathcal{S}}^{-1}_{D}[\mathbf{e}_n](\boldsymbol x)\cdot \boldsymbol{f}\mathrm{d}\sigma(\boldsymbol x)\bigg)\mathbf{e}_n\Bigg]\\
&=\mathrm{i}\gamma \sum^3_{n=1}\bigg(\int_{\partial D}\boldsymbol{\mathcal{M}}[\boldsymbol{e}_n](\boldsymbol x)\cdot \boldsymbol{f}\mathrm{d}\sigma(\boldsymbol x)\bigg)\boldsymbol{\mathcal{M}}[\mathbf{e}_n],
\end{align*}
which completes the proof.
\end{proof}

Define two $6\times6$ matrices by
\begin{align}\label{matrix:QR}
\boldsymbol{\mathscr{Q}}:=(\mathscr{Q}_{ij})_{i,j=1}^6, \quad \boldsymbol{\mathscr{R}}:=(\mathscr{R}_{ij})^6_{i,j=1},
\end{align}
where
\begin{align}
\mathscr{Q}_{ij}&:=-\int_{\partial D}\boldsymbol{\mathcal{M}}[\boldsymbol{\xi}_i]\cdot\boldsymbol{\xi}_j\mathrm{d}\sigma(\boldsymbol{x}),\label{eq:qij}\\
\mathscr{R}_{ik}&:=\sum^{3}_{n=1}c_{ni}c_{nk},\label{eq:Rij}
\end{align}
with
\begin{align}\label{eq:pij}
c_{nk}:=-\int_{\partial D}\boldsymbol{\mathcal{M}}[\mathbf{e}_n]\cdot\boldsymbol{\xi}_k\mathrm{d}\sigma(\boldsymbol{x}),\quad n=1,2,3,\quad k=1,\cdots,6,
\end{align}
and $\boldsymbol{\mathcal{M}}$ being defined by \eqref{eq:M0}.

The remainder of this subsection is to study the properties of the two matrices $\boldsymbol{\mathscr{Q}}$ and $\boldsymbol{\mathscr{R}}$.

\begin{lemm}\label{Q:symmetric}
The following properties hold:
\begin{enumerate}

\item[(i)] The matrix $\boldsymbol{\mathscr{Q}}$ is real-valued symmetric positive definite.

\item[(ii)] The matrix $\boldsymbol{\mathscr{R}}$ is real-valued symmetric positive semi-definite.

\end{enumerate}
\end{lemm}

\begin{proof}
First, we prove (i). It is evident that $\boldsymbol{\mathscr{Q}}$ is real-valued. Let $\boldsymbol{\zeta}_i:=\boldsymbol{\mathcal{S}}^{-1}_{D}[\boldsymbol{\xi}_i],i=1,\cdots,6$.  Then, we have
\begin{align*}
\mathscr{Q}_{ij}=-\int_{\partial D}\boldsymbol{\mathcal{M}}[\boldsymbol{\xi}_i]\cdot\boldsymbol{\xi}_j\mathrm{d}\sigma(\boldsymbol{x})&=-\int_{\partial D}\big(\frac{1}{2}\boldsymbol{\mathcal{I}}
+\boldsymbol{\mathcal{K}}^*_{D}\big)[\boldsymbol{\zeta}_i]\cdot\boldsymbol{\mathcal{S}}_{D}[\boldsymbol{\zeta}_j]\mathrm{d}\sigma(\boldsymbol{x}).
\end{align*}
By applying the Calder\'{o}n identity
\begin{align*}
\boldsymbol{\mathcal{K}}_D\boldsymbol{\mathcal{S}}_D=\boldsymbol{\mathcal{S}}_D\boldsymbol{\mathcal{K}}^*_D,
\end{align*}
we obtain
\begin{align*}
\mathscr{Q}_{ij}
&=-\int_{\partial D}\boldsymbol{\zeta}_i\cdot\big(\frac{1}{2}\boldsymbol{\mathcal{I}}+\boldsymbol{\mathcal{K}}_{D}\big)
\boldsymbol{\mathcal{S}}_{D}[\boldsymbol{\zeta}_j]\mathrm{d}\sigma(\boldsymbol{x})\\
&=-\int_{\partial D}\boldsymbol{\zeta}_i\cdot\boldsymbol{\mathcal{S}}_{D}\big(\frac{1}{2}\boldsymbol{\mathcal{I}}+\boldsymbol{\mathcal{K}}^*_{D}\big)
[\boldsymbol{\zeta }_j]\mathrm{d}\sigma(\boldsymbol{x})\\
&=-\int_{\partial D}\boldsymbol{\mathcal{S}}_{D}[\boldsymbol{\zeta}_i]\cdot\big(\frac{1}{2}\boldsymbol{\mathcal{I}}+\boldsymbol{\mathcal{K}}^*_{D}\big)
[\boldsymbol{\zeta }_j]\mathrm{d}\sigma(\boldsymbol{x})
=\mathscr{Q}_{ji},
\end{align*}
which indicates that $\boldsymbol{\mathscr{Q}}$ is symmetric.

It remains to verify the positive definiteness of $\boldsymbol{\mathscr{Q}}$. From \eqref{eq:xi}, we have
\begin{align*}
\big(\frac{1}{2}\boldsymbol{\mathcal{I}}+\boldsymbol{\mathcal{K}}^*_{D}\big)[\boldsymbol{\zeta}_i]=\boldsymbol{\zeta}_i,\quad i=1,\cdots,6.
\end{align*}
Furthermore, since the vectors $\boldsymbol{\xi}_i,i=1,\cdots,6,$ are linearly independent, it follows that $\sum_{i=1}^6a_i\boldsymbol{\xi}_i\neq0$ for any $\boldsymbol{a}=(a_i)^6_{i=1}\in\mathbb{R}^6\backslash\{0\}$. This leads to
\begin{align*}
\sum_{i=1}^6a_i\boldsymbol{\zeta}_i=\boldsymbol{\mathcal{S}}^{-1}_{D}\left[\sum_{i=1}^6a_i\boldsymbol{\xi}_i\right]\neq0.
\end{align*}
Consequently, combining this result with the fact that $-\boldsymbol{\mathcal{S}}_{D}$ is positive definite (cf. \cite{ando2018spectral}), we obtain that for any $\boldsymbol{a}=(a_i)^6_{i=1}\in\mathbb{R}^6\backslash\{0\}$,
\begin{align*}
\boldsymbol{a}^\top\boldsymbol{\mathscr{Q}}\boldsymbol{a}&=\sum_{i,j=1}^6a_i\mathscr{Q}_{ij}a_j
=-\sum_{i,j=1}^6a_i\bigg(\int_{\partial D}\boldsymbol{\zeta}_i\cdot\boldsymbol{\mathcal{S}}_{D}[\boldsymbol{\zeta}_j]\mathrm{d}\sigma(\boldsymbol{x})\bigg)a_j\\
&=\int_{\partial D}\bigg(\sum_{i=1}^6a_i\boldsymbol{\zeta}_i\bigg)\cdot(-\boldsymbol{\mathcal{S}}_{D})
\bigg[\sum^6_{j=1}a_j\boldsymbol{\zeta}_j\bigg]\mathrm{d}\sigma(\boldsymbol{x})\\
&=\int_{\partial D}\bigg(\sum_{i=1}^6a_i\boldsymbol{\zeta}_i\bigg)\cdot(-\boldsymbol{\mathcal{S}}_{D})
\bigg[\sum^6_{i=1}a_i\boldsymbol{\zeta}_i\bigg]\mathrm{d}\sigma(\boldsymbol{x})>0.
\end{align*}
Thus, we conclude that $\boldsymbol{\mathscr{Q}}$ is positive definite.

Next, we prove (ii).  Clearly, the matrix $\boldsymbol{\mathscr{R}}$ is real-valued and symmetric. By employing a similar argument to that used in the proof of the positive definiteness of $\boldsymbol{\mathscr{Q}}$, we deduce that for any $\boldsymbol{a}=(a_i)^6_{i=1}\in\mathbb{R}^6\backslash\{0\}$,
\begin{align*}
\boldsymbol{a}^\top\boldsymbol{\mathscr{R}}\boldsymbol{a}=\sum_{i,j=1}^6a_i\bigg(\sum_{n=1}^3c_{ni}c_{nj}\bigg)a_j
=\sum^3_{n=1}\bigg(\sum^6_{i=1}c_{ni}a_i\bigg)^2\geq0.
\end{align*}
Hence, the matrix $\boldsymbol{\mathscr{R}}$ is positive semi-definite.
\end{proof}

\section{Analysis of subwavelength resonance}\label{sec:3}

In this section, our main objective is to rewrite \eqref{Lame} within the bounded domain $D$ using the DtN map. Next, by utilizing an auxiliary sesquilinear form, we seek a necessary and sufficient condition for the well-posedness of the solution to this boundary value problem, which involves ensuring that the determinant of a specific matrix does not vanish. Based on this condition, we provide an explicit characterization of subwavelength resonances. Finally, we present the asymptotic expressions for the frequencies at which these resonances occur.

According to Definition \ref{DtN:def} regarding the DtN map, we can reformulate the scattering problem \eqref{Lame} with $\boldsymbol{\mathcal{M}}^{\sqrt{\rho}\omega}$ as a boundary value problem posed on the bounded
domain $D$:
\begin{align}\label{int:solu}
\left\{ \begin{aligned}
&\mathcal{L}_{\lambda,\mu}\boldsymbol{u}+\rho\tau^2\omega^2\boldsymbol{u}=0&&\text{in } D,\\
&\partial_{\boldsymbol{\nu}} \boldsymbol{u}=\delta\boldsymbol{\mathcal{M}}^{\sqrt{\rho}\omega}[\boldsymbol{u}-\boldsymbol{u}^{\mathrm{in}}]+\delta\partial_{\boldsymbol{\nu}}\boldsymbol{u}^{\mathrm{in}}&&\text{on }\partial D.
\end{aligned}\right.
\end{align}
The value of $\boldsymbol{u}$ in $D$ is fully determined by the system \eqref{int:solu}, whereas the value of $\boldsymbol{u}$ outside $D$ is obtained by solving an exterior elastic problem. To avoid confusion, we denote by $\boldsymbol{u}^{\mathrm{ex}}$ the solution of the elastic problem \eqref{Lame} in $\mathbb{R}^3\backslash\overline{D}$. In the regime $\omega\rightarrow0$, using the transmission conditions in \eqref{Lame} implies that for all $\boldsymbol x\in\mathbb{R}^3\backslash\overline{D}$,
\begin{align}\label{ext:solu}
\boldsymbol{u}^{\mathrm{ex}}(\boldsymbol x)=\boldsymbol{u}^{\mathrm{in}}(\boldsymbol x)+\boldsymbol{\widetilde{\boldsymbol{\mathcal{S}}}}^{\sqrt{\rho}\omega}_{D}[(\boldsymbol{\mathcal{S}}^{\sqrt{\rho}\omega}_{ D})^{-1}[\boldsymbol{u}|_{\partial D}-\boldsymbol{u}^{\mathrm{in}}|_{\partial D}]](\boldsymbol x),
\end{align}
where the trace $\boldsymbol{u}|_{\partial D}$ is defined by \eqref{int:solu}.

By multiplying a test function $\overline{\boldsymbol{v}}\in H^{1}(D)^3$ and using the identity
\begin{align*}
(\nabla \boldsymbol{u}+\nabla \boldsymbol{u}^\top):\nabla \overline{\boldsymbol{v}}=
(\nabla \boldsymbol{u}+\nabla \boldsymbol{u}^\top):\nabla \overline{\boldsymbol{v}}^\top,
\end{align*}
along with integration by parts, we obtain the variational formulation of \eqref{int:solu} as follows:
\begin{align}\label{int:varia}
\lambda\int_{D}(\nabla\cdot\boldsymbol{u})(\nabla\cdot\overline{\boldsymbol{v}})\mathrm{d}\boldsymbol{x}
+\frac{\mu}{2}\int_{D}(\nabla\boldsymbol{u}+\nabla\boldsymbol{u}^\top):
(\nabla\overline{\boldsymbol{v}}+\nabla\overline{\boldsymbol{v}}^\top)
\mathrm{d}\boldsymbol{x}-\rho \tau^{2}\omega^2\int_{D}\boldsymbol{u}\cdot\overline{\boldsymbol{v}}\mathrm{d}\boldsymbol{x}\nonumber\\
-\delta\int_{\partial D}\boldsymbol{\mathcal{M}}^{\sqrt{\rho}\omega}
[\boldsymbol{u}]\cdot\overline{\boldsymbol{v}}\mathrm{d}\sigma(\boldsymbol{x})
=\delta\int_{\partial D}\big(\partial_{\boldsymbol{\nu}} \boldsymbol{u}^{\mathrm{in}}-\boldsymbol{\mathcal{M}}^{\sqrt{\rho}\omega}
[\boldsymbol{u}^{\mathrm{in}}]\big)\cdot\overline{\boldsymbol{v}}\mathrm{d}\sigma(\boldsymbol{x}).
\end{align}
 The notation $A:B=\mathrm{tr}(AB^{\top})$ represents the Frobenius product of square matrices $A$ and $B$.

\subsection{An auxiliary sesquilinear form}

For $\boldsymbol{u},\boldsymbol{v}\in H^{1}(D)^{3}$, we define a sesquilinear form by adding the sesquilinear form $\sum_{i=1}^6\int_{D}\boldsymbol{u}\cdot\boldsymbol{\xi}_i\mathrm{d}\boldsymbol{x}\int_{D}\overline{\boldsymbol{v}}\cdot\boldsymbol{\xi}_i\mathrm{d}\boldsymbol{x}$ to the left-hand side of \eqref{int:varia} as follows:
\begin{align}\label{a(u,v)}
a_{\omega,\delta}(\boldsymbol{u},\boldsymbol{v})&:=a_{0,0}(\boldsymbol{u},\boldsymbol{v})-\rho\tau^{2}\omega^2\int_{D}\boldsymbol{u}\cdot\overline{\boldsymbol{v}}\mathrm{d}\boldsymbol{x}-\delta\int_{\partial D}\boldsymbol{\mathcal{M}}^{\sqrt{\rho}\omega}[\boldsymbol{u}]\cdot\overline{\boldsymbol{v}}\mathrm{d}\sigma(\boldsymbol{x}),
\end{align}
where
\begin{align*}
a_{0,0}(\boldsymbol{u},\boldsymbol{v})&:=\lambda\int_{D}(\nabla\cdot\boldsymbol{u})(\nabla\cdot\overline{\boldsymbol{v}})\mathrm{d}\boldsymbol{x}
+\frac{\mu}{2}\int_{D}(\nabla\boldsymbol{u}+\nabla\boldsymbol{u}^\top):(\nabla\overline{\boldsymbol{v}}+\nabla\overline{\boldsymbol{v}}^\top)
\mathrm{d}\boldsymbol{x}\nonumber\\
&\quad+\sum_{i=1}^6\int_{D}\boldsymbol{u}\cdot\boldsymbol{\xi}_i\mathrm{d}\boldsymbol{x}\int_{D}\overline{\boldsymbol{v}}\cdot\boldsymbol{\xi}_i\mathrm{d}\boldsymbol{x}.
\end{align*}
Observe that the sesquilinear form $a_{\omega,\delta}$ is an analytic perturbation in $\omega$ and $\delta$ of  $a_{0,0}$.

\begin{lemm}\label{lemma coer}
The sesquilinear form $a_{\omega,\delta}$, defined by \eqref{a(u,v)}, is bounded and coercive.
\end{lemm}

\begin{proof}
We claim that $a_{0,0}$ is bounded and coercive. By \eqref{DtN-NP}, the DtN operator $\boldsymbol{\mathcal{M}}^{\sqrt{\rho}\omega}:H^{\frac{1}{2}}(\partial D)^3\rightarrow H^{-\frac{1}{2}}(\partial D)^3$ is bounded due to the boundedness of $(\boldsymbol{\mathcal{S}}^{\sqrt{\rho}\omega}_{D})^{-1}$ and $\boldsymbol{\mathcal{K}}^{\sqrt{\rho}\omega,*}_{D}$ as $\omega\rightarrow0$. Therefore, by standard perturbation theory, it follows that $a_{\omega,\delta}$ remains bounded and coercive for sufficiently small $\delta$ and $\omega$. The remainder of the proof is to verify this claim.

First, we show that $a_{0,0}$ is bounded. By the Cauchy--Schwarz inequality, there exists a constant $C_1>0$, dependent on $\lambda$ and $\mu$, such that
\begin{align*}
|a_{0,0}(\boldsymbol{u},\boldsymbol{v})|&\leq \frac{\mu}{2}\|\nabla\boldsymbol{u}+\nabla\boldsymbol{u}^\top\|_{L^{2}(D)^{3\times3}}\|\nabla\boldsymbol{v}+\nabla\boldsymbol{v}^\top\|_{L^{2}(D)^{3\times3}}
\\ &\quad+|\lambda| \|\nabla\cdot\boldsymbol{u}\|_{L^{2}(D)}\|\nabla\cdot\boldsymbol{v}\|_{L^{2}(D)}+6\|\boldsymbol{u}\|_{L^2(D)^3}\|\boldsymbol{v}\|_{L^2(D)^3}
\\
&\leq C_{1}\|\boldsymbol{u}\|_{H^{1}(D)^{3}}\|\boldsymbol{v}\|_{H^{1}(D)^{3}},
\end{align*}
which demonstrates the boundedness of $a_{0,0}$.

Next, we show the coercivity of $a_{0,0}$. We rewrite $a_{0,0}(\boldsymbol{u},\boldsymbol{u})$ as
\begin{align*}
a_{0,0}(\boldsymbol{u},\boldsymbol{u})=2\mu\sum_{i,j=1}^3\|\mathcal{E}_{ij}(\boldsymbol{u})\|^2_{L^2(D)}+
\lambda\|\nabla\cdot\boldsymbol{u}\|^2_{L^2(D)}+\sum_{i=1}^6\Big|\int_{D}\boldsymbol{u}\cdot\boldsymbol{\xi}_i \mathrm{d}\boldsymbol{x}\Big|^2,
\end{align*}
where $\mathcal{E}_{ij}(\boldsymbol{u})$ is defined in \eqref{E:Eij}. Since $3\lambda+2\mu>0$ (cf. \eqref{convexity}), there exists a constant $\kappa_0:=\min\{2\mu, 3\lambda+2\mu\}>0$ such that
\begin{align}\label{coer}
2\mu\sum_{i,j=1}^3\|\mathcal{E}_{ij}(\boldsymbol{u})\|^2_{L^2(D)}+
\lambda\|\nabla\cdot\boldsymbol{u}\|^2_{L^2(D)}\geq \kappa_0\sum_{i,j=1}^3\|\mathcal{E}_{ij}(\boldsymbol{u})\|^2_{L^2(D)}.
\end{align}
In fact, if $\lambda\geq 0$, then \eqref{coer} holds with $\kappa_0=2\mu$.

If $-\frac{2}{3}\mu<\lambda<0$, then
\begin{align*}
&2\mu\sum_{i,j=1}^3\|\mathcal{E}_{ij}(\boldsymbol{u})\|^2_{L^2(D)}+
\lambda\|\nabla\cdot\boldsymbol{u}\|^2_{L^2(D)}\\
&\geq\int_D(2\mu+3\lambda)(|\partial_{x_1}u_1|^2+|\partial_{x_2}u_2|^2+|\partial_{x_3}u_3|^2)\mathrm{d}\boldsymbol{x}\\
&\quad +\mu\int_D(|\partial_{x_1}u_2+\partial_{x_1}u_2|^2+|\partial_{x_1}u_3+\partial_{x_3} u_1|^2+|\partial_{x_2}u_3+\partial_{x_3}u_2|^2)\mathrm{d}\boldsymbol{x}\nonumber\\
&\geq\left(3\lambda+2\mu\right)\sum_{i,j=1}^3\|\mathcal{E}_{ij}(\boldsymbol{u})\|^2_{L^2(D)}.
\end{align*}
Additionally, if there exists a constant $C_2>0$ such that for any $\boldsymbol{u}\in H^{1}(D)^{3}$,
\begin{align}\label{add}
C_2\bigg(\sum_{i,j=1}^3\|\mathcal{E}_{ij}(\boldsymbol{u})\|^2_{L^2(D)}+\frac{1}{\kappa_0}\sum_{i=1}^6\Big|\int_{D}\boldsymbol{u}\cdot\boldsymbol{\xi}_i \mathrm{d}\boldsymbol{x}\Big|^2\bigg)\geq\|\boldsymbol{u}\|^{2}_{H^{1}(D)^{3}}.
\end{align}
Combining the above inequality with \eqref{coer}, we obtain
\begin{align*}
a_{0,0}(\boldsymbol{u},\boldsymbol{u})\geq \kappa_0C^{-1}_2\|\boldsymbol{u}\|^{2}_{H^{1}(D)^{3}},
\end{align*}
which confirms the coercivity of $a_{0,0}$.

Let us now verify \eqref{add} by proceeding via contradiction. Suppose that for any constant $C>0$, there exists some $\boldsymbol{u}\in H^{1}(D)^{3}$ satisfying
\begin{align*}
\sum_{i,j=1}^3\|\mathcal{E}_{ij}(\boldsymbol{u})\|^2_{L^2(D)}
+\frac{1}{\kappa_0}\sum_{i=1}^6\Big|\int_{D}\boldsymbol{u}\cdot\boldsymbol{\xi}_i \mathrm{d}\boldsymbol{x}\Big|^2<\frac{1}{C}\|\boldsymbol{u}\|^{2}_{H^{1}(D)^{3}}.
\end{align*}
Then, there exits a sequence $\{\boldsymbol{u}_n\}\in H^1(D)^3$ satisfying
$\|\boldsymbol{u}_n\|_{H^1(D)^3}=1$ and
\begin{align}\label{E:converse}
\sum_{i,j=1}^3\|\mathcal{E}_{ij}(\boldsymbol{u}_n)\|^2_{L^2(D)}
+\frac{1}{\kappa_0}\sum_{i=1}^6\Big|\int_{D}\boldsymbol{u}_n\cdot\boldsymbol{\xi}_i \mathrm{d}\boldsymbol{x}\Big|^2\rightarrow 0
\end{align}
as $n\to +\infty$.
Since $H^{1}(D)^{3}$ is compactly embedded into $L^2(D)^3$, there exits a subsequence $\{\boldsymbol{u}_{n_l}\}\subset\{\boldsymbol{u}_{n}\}$ such that $\{\boldsymbol{u}_{n_l}\}$ is a Cauchy sequence in $L^2(D)^3$. By Korn's inequality (cf. \cite[p.299]{Mclean2000strongly}), there exists a constant $C_3>0$ such that
\begin{align*}
\|\boldsymbol{u}_{n_l}\|^{2}_{H^{1}(D)^{3}}\leq C_3\bigg(\sum_{i,j=1}^3\|\mathcal{E}_{ij}(\boldsymbol{u}_{n_l})\|^2_{L^2(D)} +\|\boldsymbol{u}_{n_l}\|^2_{L^2(D)^3}\bigg).
\end{align*}
It follows that
\begin{align*}
\|\boldsymbol{u}_{n_l}-\boldsymbol{u}_{n_m}\|^{2}_{H^{1}(D)^{3}}
&\leq 2C_3\big(\sum_{i,j=1}^3\|\mathcal{E}_{ij}(\boldsymbol{u}_{n_l})\|^2_{L^2(D)}\\
&\quad +\sum_{i,j=1}^3\|\mathcal{E}_{ij}(\boldsymbol{u}_{n_m})\|^2_{L^2(D)}+
\|\boldsymbol{u}_{n_l}-\boldsymbol{u}_{n_m}\|^2_{L^2(D)^3}\big)\rightarrow 0
\end{align*}
as $l,m\rightarrow+\infty$.
Therefore, $\{\boldsymbol{u}_{n_l}\}$ is a Cauchy sequence in $H^1(D)^3$, and consequently,
\begin{align*}
\|\boldsymbol{u}_{n_l}-\boldsymbol{u}\|_{H^1(D)^3}\rightarrow 0, \quad l\rightarrow \infty,
\end{align*}
for some $\boldsymbol{u}=(u_i)_{i=1}^3$ and $\|\boldsymbol{u}\|_{H^1(D)^3}=1$. Using \eqref{E:converse}, we conclude that
\begin{align*}
\left\{ \begin{aligned}
&\frac{1}{2}(\partial_{x_i}u_j+\partial_{x_j}u_i)=0,\\
&\sum_{i=1}^6\Big|\int_{D}\boldsymbol{u}\cdot\boldsymbol{\xi}_i \mathrm{d}\boldsymbol{x}\Big|^2=0,
\end{aligned}\right.
 \end{align*}
which implies that the only solution is $\boldsymbol{u}=0$. However, this contradicts the fact that $\|\boldsymbol{u}\|_{H^1(D)^3}=1$. Therefore,  \eqref{add} holds and the proof is completed.
\end{proof}

\subsection{Characterization of subwavelength resonances}

This subsection is to establish the necessary and sufficient conditions for the existence and uniqueness of solutions to the interior problem \eqref{int:solu}. These conditions provide an explicit characterization of subwavelength resonances.

The definition of subwavelength resonance is provided below.

\begin{defi}\label{D:reson}
A subwavelength resonance of equation \eqref{int:solu} is defined as any $\omega\in\mathbb{C}$ that satisfies
$\omega(\delta)\rightarrow 0$ as $\delta\rightarrow 0^+$, and such that the following system
\begin{align}\label{int:reson}
\left\{ \begin{aligned}
&\mathcal{L}_{\lambda,\mu}\boldsymbol{u}(\omega,\delta)+\rho\tau^2\omega^2\boldsymbol{u}(\omega,\delta)=0&&\text{in } D,\\
&\partial_{\boldsymbol{\nu}}\boldsymbol{u}(\omega,\delta)=\delta\boldsymbol{\mathcal{M}}^{\sqrt{\rho}\omega}\left[\boldsymbol{u}(\omega,\delta)\right]&&\text{on }\partial D
\end{aligned}\right.
\end{align}
admits a nontrivial solution $\boldsymbol{u}(\omega,\delta)\in H^1(D)^3$.
\end{defi}

\begin{rema}\label{resonance}
By taking the complex conjugate on both sides of \eqref{int:reson}, and noting that  $\overline{\boldsymbol{\mathcal{M}}^{\sqrt{\rho}\omega}[\boldsymbol{f}]} =\boldsymbol{\mathcal{M}}^{-\sqrt{\rho}\overline{\omega}}[\overline{\boldsymbol{f}}]$,
it follows that $-\overline {\omega}$ is also a resonant frequency with identical multiplicities in the case of
$\omega$ being a subwavelength resonant frequency.
 Moreover, the corresponding nontrivial solution is $\overline{\boldsymbol{u}}(\omega,\delta)$.
\end{rema}

If subwavelength resonance occurs, we employ the notation $\omega^{\pm}$ to denote the resonant frequencies, where $\Re(\omega^{+})>0$ and $\omega^{-}=-\overline{\omega^+}$.

By taking $\delta=0$ in \eqref{int:reson}, we obtain a resonant frequency $\omega=0$ corresponding to the function $u(0,0):=\chi_D\boldsymbol{\xi}_i$, where $\boldsymbol{\xi}_i$ is a solution of \eqref{Lame:traction 0} and $\chi_D$ denotes the characteristic function of the set $D$, which is defined as
\begin{align*}
\chi_D=\left\{ \begin{aligned}
&0,&&\boldsymbol x\in\mathbb{R}^3\backslash \overline D,\\
&1,&&\boldsymbol x\in D.
\end{aligned}\right.
\end{align*}
Thus, it suffices to study the perturbation and the splitting of the nonlinear eigenvalue $\omega=0$ with $\delta=0$ in order to characterize the resonances.

Now, we estimate the leading terms of the resonant frequencies from the following variational formulation:
\begin{align}\label{reson:leading equa}
&\lambda\int_{D}(\nabla\cdot\boldsymbol{u})(\nabla\cdot\overline{\boldsymbol{v}})\mathrm{d}\boldsymbol{x}
+\frac{\mu}{2}\int_{D}(\nabla\boldsymbol{u}+\nabla\boldsymbol{u}^\top):(\nabla\overline{\boldsymbol{v}}+\nabla\overline{\boldsymbol{v}}^\top)
\mathrm{d}\boldsymbol{x}-\rho\tau^2\omega^2\int_D\boldsymbol{u}\cdot\overline{\boldsymbol{v}}\mathrm{d}\boldsymbol{x}\nonumber\\
&=\delta\int_{\partial D}\boldsymbol{\mathcal{M}}^{\sqrt{\rho}\omega}[\boldsymbol{u}]\cdot\overline{\boldsymbol{v}}\mathrm{d}\sigma(\boldsymbol{x})
\end{align}
for all $\boldsymbol{v}\in H^{1}(D)^3$, which corresponds to  \eqref{int:reson}.

Since $\boldsymbol{\xi}_i,i=1,\cdots,6,$ are solutions of {\eqref{Lame:traction 0}}, it follows from integration by parts that
\begin{align*}
&\lambda\int_{D}(\nabla\cdot\boldsymbol{\xi}_i)(\nabla\cdot\boldsymbol{\xi}_j)\mathrm{d}\boldsymbol{x}
+\frac{\mu}{2}\int_{D}(\nabla\boldsymbol{\xi}_i+\nabla\boldsymbol{\xi}_i^\top):(\nabla\boldsymbol{\xi}_j+\nabla\boldsymbol{\xi}_j^\top) \mathrm{d}\boldsymbol{x}\nonumber\\
&=-\int_D\mathcal{L}_{\lambda,\mu}\boldsymbol{\xi}_i\cdot\boldsymbol{\xi}_j\mathrm{d}\boldsymbol{x}+\int_{\partial D}\partial_{\boldsymbol{\nu}} \boldsymbol{\xi}_i\cdot\boldsymbol{\xi}_j\mathrm{d}\sigma(\boldsymbol{x})=0.
\end{align*}
Thus, if  taking $\boldsymbol{v}=\boldsymbol{\xi}_j$, $\boldsymbol{u}(\omega,\delta)\approx\boldsymbol{\xi}_i$
in \eqref{reson:leading equa},  and  applying Proposition \ref{DtNasy}, we can  obtain
\begin{align*}
-\rho\tau^2\omega^2\int_D\boldsymbol{\xi}_i\cdot\boldsymbol{\xi}_j\mathrm{d}\boldsymbol{x}\approx\delta\int_{\partial D} \boldsymbol{\mathcal{M}}[\boldsymbol{\xi}_i]\cdot\boldsymbol{\xi}_j\mathrm{d}\sigma(\boldsymbol{x}),\quad i,j=1,\cdots,6.
\end{align*}
Due to the orthogonality of $\{\boldsymbol{\xi}_i\}_{i=1}^6$, it follows that
\begin{align*}
-\rho\tau^2\omega^2\boldsymbol{\mathscr{I}}+\delta \boldsymbol{\mathscr{Q}}\approx 0,
\end{align*}
where $\boldsymbol{\mathscr{I}}$ is the $6\times 6$ identity matrix and $\boldsymbol{\mathscr{Q}}$ is defined by \eqref{matrix:QR}.

Denote by $\lambda_i$ and $\boldsymbol{\mathscr{V}}_i=(\mathscr{V}_{ij})_{j=1}^6$ the eigenvalues and the corresponding eigenvectors of $\boldsymbol{\mathscr{Q}}$, respectively, with
\begin{align}\label{eigen:v}
\boldsymbol{\mathscr{V}}_i^{\top}\boldsymbol{\mathscr{V}}_i=1,\quad \boldsymbol{\mathscr{V}}_i^{\top}\boldsymbol{\mathscr{V}}_j=0,
\quad i\neq j.
\end{align}
Define a $6\times6$ unitary matrix $\boldsymbol{\mathscr{V}}$ by
\begin{align*}
\boldsymbol{\mathscr{V}}:=(\boldsymbol{\mathscr{V}}_{1},\cdots,\boldsymbol{\mathscr{V}}_6).
\end{align*}
Since $\boldsymbol{\mathscr{Q}}$ is a real-valued symmetric matrix (cf. Lemma \ref{Q:symmetric}), we derive
\begin{align*}
\boldsymbol{\mathscr{V}}^{\top}(-\rho\tau^2\omega^2\boldsymbol{\mathscr{I}}+\delta \boldsymbol{\mathscr{Q}})\boldsymbol{\mathscr{V}}=\mathrm{diag}(-\rho\tau^2\omega^2+\delta\lambda_i)^6_{i=1}\approx0,
\end{align*}
which leads to
\begin{align*}
\omega^{\pm}_i(\delta^{\frac12})\approx\pm\sqrt{\frac{\delta\lambda_i}{\rho\tau^2}}=\pm\sqrt{\frac{\epsilon\lambda_i}{\rho}},\quad i=1,\cdots,6.
\end{align*}

When no resonance occurs, we aim to investigate the well-posedness of the solution to \eqref{int:solu}. In fact, solving \eqref{int:solu} is equivalent to finding a weak solution to the associated variational problem \eqref{int:varia}. Our objective is to provide a necessary and sufficient condition for the existence and uniqueness of the weak solution to this variational problem. Before proceeding, we introduce some auxiliary definitions and preliminary results to facilitate the analysis.

By applying Lemma \ref{lemma coer} and the Lax--Milgram theorem, for each $i$, there exists a unique solution $\boldsymbol{w}_i(\omega,\delta)$ such that for $i=1,\cdots,6$,
\begin{align}\label{wi}
a_{\omega,\delta}(\boldsymbol{w}_i(\omega,\delta),\boldsymbol{v})=\int_D\overline{\boldsymbol{v}}\cdot \boldsymbol{\xi}_{i} \mathrm{d}\boldsymbol{x}.
\end{align}
Similarly, there exits a unique solution $\boldsymbol{w_f}(\omega,\delta)$ satisfying
\begin{align}\label{wf}
a_{\omega,\delta}(\boldsymbol{w}_{\boldsymbol f}(\omega,\delta),\boldsymbol{v})=\delta\int_{\partial D}\big(\partial_{\boldsymbol{\nu}} \boldsymbol{u}^{\mathrm{in}}-\boldsymbol{\mathcal{M}}^{\sqrt{\rho}\omega}
[\boldsymbol{u}^{\mathrm{in}}]\big)\cdot\overline{\boldsymbol{v}}\mathrm{d}\sigma(\boldsymbol{x}).
\end{align}
Here, $\boldsymbol{w}_i(\omega,\delta)$ and $\boldsymbol{w}_{\boldsymbol f}(\omega,\delta)$ are analytic in both $\omega$ and $\delta$. Define two $6$-component vectors and a $6\times6$ matrix as follows:
\begin{align}
\boldsymbol{s}(\omega,\delta)&:=(s_i(\omega,\delta))^6_{i=1},\quad s_i(\omega,\delta)=\int_D\boldsymbol{u}(\omega,\delta)\cdot \boldsymbol{\xi}_{i}\mathrm{d}\boldsymbol{x},\label{eq:psi}\\
\boldsymbol{g}(\omega,\delta)&:=(g_i(\omega,\delta))^6_{i=1}, \quad  g_i(\omega,\delta)=\int_D \boldsymbol{w}_{\boldsymbol f}(\omega,\delta)\cdot \boldsymbol{\xi}_i\mathrm{d}\boldsymbol{x},\label{eq:g}\\
\boldsymbol{\mathscr{A}}(\omega,\delta)&:=(\mathscr{A}_{ij}(\omega,\delta))^6_{i,j=1},\quad \mathscr{A}_{ij}(\omega,\delta)=\int_D \boldsymbol{w}_j(\omega,\delta)\cdot \boldsymbol{\xi}_i\mathrm{d}\boldsymbol{x}.\label{eq:b}
\end{align}
It is clear to note that $\boldsymbol{g}(\omega,\delta)$ and $\boldsymbol{\mathscr{A}}(\omega,\delta)$ are uniquely determined by $\boldsymbol{w}_{\boldsymbol f}(\omega,\delta)$ and $\boldsymbol{w}_j(\omega,\delta)$, respectively.

\begin{lemm}
Let $\omega\in\mathbb{C}$ and $\delta\in\mathbb{R}^+$ be in a neighborhood of zero. The variational problem \eqref{int:varia} has a unique solution $\boldsymbol{u}(\omega,\delta)\in H^{1}(D)^3$ if the following condition holds:
\begin{align}\label{E:unique}
\det\left(\boldsymbol{\mathscr{I}}-\boldsymbol{\mathscr{A}}(\omega,\delta)\right)\neq0.
\end{align}
Moreover, under the condition \eqref{E:unique}, the unique solution $\boldsymbol{u}(\omega,\delta)$ of the variational problem \eqref{int:varia} can be expressed as
\begin{align}\label{solu:express}
\boldsymbol{u}(\omega,\delta)=\sum^{6}_{j=1}s_{j}(\omega,\delta)\boldsymbol{w}_j(\omega,\delta)
+\boldsymbol{w}_{\boldsymbol f}(\omega,\delta),
\end{align}
where $\boldsymbol{w}_j(\omega,\delta)$ and $\boldsymbol{w}_{\boldsymbol f}(\omega,\delta)$ are the unique solutions to the variational problems \eqref{wi} and \eqref{wf}, respectively, and $\boldsymbol{s}(\omega,\delta)$ satisfies the following system:
\begin{align}\label{matrix}
\left(\boldsymbol{\mathscr{I}}-\boldsymbol{\mathscr{A}}(\omega,\delta)\right)\boldsymbol{s}(\omega,\delta)=\boldsymbol{g}(\omega,\delta).
\end{align}
\end{lemm}

\begin{proof}
It is clear that the variational problem \eqref{int:varia} is equivalent to
\begin{align*}
a_{\omega,\delta}(\boldsymbol{u},\boldsymbol{v})-\sum_{i=1}^6\int_{D}\boldsymbol{u}
\cdot\boldsymbol{\xi}_i\mathrm{d}\boldsymbol{x}\int_{D}\overline{\boldsymbol{v}}\cdot\boldsymbol{\xi}_i
\mathrm{d}\boldsymbol{x}
=\delta\int_{\partial D}\big(\partial_{\boldsymbol{\nu}}\boldsymbol{u}^{\mathrm{in}}-\boldsymbol{\mathcal{M}}^{\sqrt{\rho}\omega}
[\boldsymbol{u}^{\mathrm{in}}]\big)\cdot\overline{\boldsymbol{v}}\mathrm{d}\sigma(\boldsymbol{x}).
\end{align*}
Using \eqref{wi} and \eqref{wf} yields
\begin{align*}
a_{\omega,\delta}(\boldsymbol{u}(\omega,\delta),\boldsymbol{v})-\sum_{i=1}^6\Big(\int_{D}\boldsymbol{u}(\omega,\delta)
\cdot\boldsymbol{\xi}_i\mathrm{d}\boldsymbol{x}\Big)a_{\omega,\delta}(\boldsymbol{w}_i(\omega,\delta),\boldsymbol{v})
=a_{\omega,\delta}(\boldsymbol{w}_{\boldsymbol f}(\omega,\delta),\boldsymbol{v}),
\end{align*}
which leads to
\begin{align*}
\boldsymbol{u}(\omega,\delta)-\sum_{i=1}^6\Big(\int_D\boldsymbol{u}(\omega,\delta)\cdot \boldsymbol{\xi}_{i} \mathrm{d}\boldsymbol{x}\Big)\boldsymbol{w}_i(\omega,\delta)=\boldsymbol{w}_{\boldsymbol f}(\omega,\delta).
\end{align*}
Thus, we derive
\begin{align*}
\int_D\boldsymbol{u}(\omega,\delta)\cdot \boldsymbol{\xi}_{i} \mathrm{d}\boldsymbol{x}
-\sum^{6}_{j=1}\int_D\boldsymbol{w}_j(\omega,\delta)\cdot \boldsymbol{\xi}_{i} \mathrm{d}\boldsymbol{x}\int_D\boldsymbol{u}(\omega,\delta)\cdot \boldsymbol{\xi}_{j} \mathrm{d}\boldsymbol{x}\\
=\int_D\boldsymbol{w}_{\boldsymbol f}(\omega,\delta)\cdot \boldsymbol{\xi}_{i} \mathrm{d}\boldsymbol{x},\quad i=1,\cdots,6,
\end{align*}
which shows that $\boldsymbol{s}(\omega,\delta)$ satisfies \eqref{matrix}. Consequently, the variational problem \eqref{int:varia} has a unique solution $\boldsymbol{u}(\omega,\delta)$  given by \eqref{solu:express} if and only if the condition \eqref{E:unique} holds.
\end{proof}

According to  Definition \ref{D:reson}, the variational problem \eqref{wf} admits a unique solution $\boldsymbol{w}_{\boldsymbol f}(\omega,\delta)=0$, when subwavelength resonance occurs, implying that $\boldsymbol{g}=0$. Therefore,
the resonant frequencies, denoted as $\omega = \omega(\delta^{\frac{1}{2}}) \in \mathbb{C}$, are characterized by the condition
\begin{align*}
\det\left(\boldsymbol{\mathscr{I}}-\boldsymbol{\mathscr{A}}(\omega,\delta)\right)=0.
\end{align*}
In other words, the resonant frequency $\omega$ is the characteristic value of the operator  $\boldsymbol{\mathscr{I}}-\boldsymbol{\mathscr{A}}(\omega,\delta)$, where the definition of the characteristic value of an analytic operator-valued function can be found in \cite[p.9]{Lay}.

Based on the relationship \eqref{eq:b} between $\boldsymbol{\mathscr{A}}(\omega,\delta)$ and $\boldsymbol{w}_i(\omega,\delta)$, the following proposition  provides the asymptotic expansions of $\boldsymbol{w}_i(\omega,\delta),i=1,\cdots,6,$ with respect to $\omega$ and $\delta$.

\begin{prop}\label{wi:expansion}
Let $\omega\in\mathbb{C}$ and  $\delta\in\mathbb{R}^+$. For each $i=1,\cdots,6$, the solution $\boldsymbol{w}_i(\omega,\delta)$ of \eqref{wi} admits the asymptotic expansion as $\omega\rightarrow0$ and $\delta\rightarrow 0$:
\begin{align}\label{wjasy}
\boldsymbol{w}_i(\omega,\delta)&=\chi_D\boldsymbol{\xi}_i+\omega^2\rho\tau^2\chi_D\boldsymbol{\xi}_i-
\delta\chi_D\sum_{k=1}^6\mathscr{Q}_{ik}\boldsymbol{\xi}_k+\delta\widetilde{\boldsymbol{w}}_{i;0,1}
+\mathrm{i}\omega\delta\sqrt{\rho}\gamma\chi_D\sum_{k=1}^6\mathscr{R}_{ik}\boldsymbol{\xi}_k\nonumber\\
&\quad+\mathrm{i}\omega\delta\sqrt{\rho}\gamma\widetilde{\boldsymbol{w}}_{i;1,1}
+\mathcal{O}\big((\omega^2+\delta\big)^2),
\end{align}
where $\gamma,\mathscr{Q}_{ik}$, and $\mathscr{R}_{ik}$ are defined in \eqref{ri}, \eqref{eq:qij}, and \eqref{eq:Rij}, respectively. Moreover, the function $\widetilde{\boldsymbol{w}}_{i;0,1}$ satisfies the following elastic problem:
\begin{align}\label{wj01}
\left\{ \begin{aligned}
&-\mathcal{L}_{\lambda,\mu}\widetilde{\boldsymbol{w}}_{i;0,1}=\sum^6_{k=1}\mathscr{Q}_{ik}\boldsymbol{\xi}_k&& \text{\rm in } D,\\
&\partial_{\boldsymbol{\nu}}\widetilde{\boldsymbol{w}}_{i;0,1}=\boldsymbol{\mathcal{M}}[\boldsymbol{\xi}_i]&& \text{\rm on } \partial D,\\
&\int_D\widetilde{\boldsymbol{w}}_{i;0,1}\cdot\boldsymbol{\xi}_k\mathrm{d}\boldsymbol{x}=0,&& k=1,\cdots,6
\end{aligned}\right.
\end{align}
and the function $\widetilde{\boldsymbol{w}}_{i;1,1}$ satisfies the following elastic problem:
\begin{align}\label{E:wj11}
\left\{ \begin{aligned}
&-\mathcal{L}_{\lambda,\mu}\widetilde{\boldsymbol{w}}_{i;1,1}+\sum^6_{k=1}\mathscr{R}_{ik}\boldsymbol{\xi}_k=0&& \text{\rm in } D,\\
&\partial_{\boldsymbol{\nu}}\widetilde{\boldsymbol{w}}_{i;1,1}=-\sum_{n=1}^3c_{ni}\boldsymbol{\mathcal{M}}[\mathbf{e}_n]&& \text{\rm on } \partial D,\\
&\int_D\widetilde{\boldsymbol{w}}_{i;1,1}\cdot\boldsymbol{\xi}_k\mathrm{d}\boldsymbol{x}=0,&& k=1,\cdots,6,
\end{aligned}\right.
\end{align}
where  $c_{ni}$ are defined by \eqref{eq:pij}.
\end{prop}

\begin{proof}
By combining  the definition \eqref{a(u,v)} of $a_{\omega,\delta}({\boldsymbol{u},\boldsymbol{v}})$ with the  variational problem \eqref{wi}, we derive the following strong form:
\begin{align}\label{wjPDE}
\left\{ \begin{aligned}
&-\mathcal{L}_{\lambda,\mu}\boldsymbol{w}_{i}(\omega,\delta)+\sum^{6}_{i=1}\bigg(\int_D\boldsymbol{w}_i(\omega,\delta)\cdot \boldsymbol{\xi}_j\mathrm{d}\boldsymbol{x}\bigg)\boldsymbol{\xi}_j-\rho\tau^2\omega^2\boldsymbol{w}_i(\omega,\delta)
=\boldsymbol{\xi}_i&& \text{in } D,\\
&\partial_{\boldsymbol{\nu}}\boldsymbol{w}_i(\omega,\delta)=\delta\boldsymbol{\mathcal{M}}^{\sqrt{\rho}\omega}[\boldsymbol{w}_i(\omega,\delta)]&&\text{on }\partial D.
\end{aligned}\right.
\end{align}
Since $\boldsymbol{w}_i(\omega,\delta)$ is analytic with respect to both $\omega$ and $\delta$, there exist functions $\{\boldsymbol{w}_{i;l,k}\}_{l,k=0}^{+\infty}$ such that the power series
\begin{align}\label{wjpk}
\boldsymbol{w}_i(\omega,\delta)=\sum^{+\infty}_{l,k=0}\omega^{l}\delta^{k}\boldsymbol{w}_{i;l,k}
\end{align}
is convergent in $H^1(D)^3$. Moreover, it follows from Proposition \ref{DtNasy} that there are operators $\{\boldsymbol{\mathcal{M}}_{m}\}_{m=1}^{+\infty}$ such that $\boldsymbol{\mathcal{M}}^{\sqrt{\rho}\omega}$ can be decomposed as the following power series:
\begin{align}\label{Se:DtN}
\boldsymbol{\mathcal{M}}^{\sqrt{\rho}\omega}=\sum_{m=0}^{+\infty}\rho^{\frac{m}{2}}\omega^m\boldsymbol{\mathcal{M}}_{m},
\end{align}
where $\boldsymbol{\mathcal{M}}_0:=\boldsymbol{\mathcal{M}}$, with $\boldsymbol{\mathcal{M}}$ given by \eqref{eq:M0}, and $\boldsymbol{\mathcal{M}}_1$ is defined by \eqref{eq:M1}. This series is convergent in the space of operators from  $H^{\frac12}(\partial D)^3$ to $H^{-\frac12}(\partial D)^3$.

Assuming that $\boldsymbol{w}_{i;l,k}=0$ for $l,k<0$ and substituting \eqref{wjpk} and \eqref{Se:DtN} into \eqref{wjPDE}, we obtain
\begin{align*}
\left\{ \begin{aligned}
&-\sum^{+\infty}_{l,k=0}\omega^l\delta^k\mathcal{L}_{\lambda,\mu}\boldsymbol{w}_{i;l,k}+
\sum^{+\infty}_{l,k=0}\omega^l\delta^k\sum^{6}_{j=1}\bigg(\int_D\boldsymbol{w}_{i;l,k}\cdot \boldsymbol{\xi}_j\mathrm{d}\boldsymbol{x}\bigg)\boldsymbol{\xi}_j
-\sum^{+\infty}_{l,k=0}\omega^l\delta^k\rho\tau^2\boldsymbol{w}_{i;l-2,k}=\boldsymbol{\xi}_i&& \text{in } D,\\
&\sum^{+\infty}_{l,k=0}\omega^l\delta^k\partial_{\boldsymbol{\nu}}\boldsymbol{w}_{i;l,k}=\sum^{+\infty}_{l,k=0}\omega^l\delta^k\sum^l_{m=0}\rho^{\frac{m}{2}}\boldsymbol{\mathcal{M}}_m[\boldsymbol{w}_{i;l-m,k-1}]&& \text{on }\partial D.
\end{aligned}\right.
\end{align*}
Each coefficient of $\omega^l\delta^k$ must vanish independently, leading to
\begin{align*}
\left\{ \begin{aligned}
&-\mathcal{L}_{\lambda,\mu}\boldsymbol{w}_{i;l,k}+\sum^{6}_{j=1}\bigg(\int_D\boldsymbol{w}_{i;l,k}\cdot \boldsymbol{\xi}_j\mathrm{d}\boldsymbol{x}\bigg)\boldsymbol{\xi}_j=\rho\tau^2\boldsymbol{w}_{i;l-2,k}
+\boldsymbol{\xi}_i\delta_{l=0}\delta_{k=0}&& \text{in } D,\\
&\partial_{\boldsymbol{\nu}}\boldsymbol{w}_{i;l,k}=\sum^l_{m=0}\rho^{\frac{m}{2}}\boldsymbol{\mathcal{M}}_m[\boldsymbol{w}_{i;l-m,k-1}]&& \text{on }\partial D.
\end{aligned}\right.
\end{align*}

For $l=k=0$, we have
\begin{align*}
\left\{ \begin{aligned}
&-\mathcal{L}_{\lambda,\mu}\boldsymbol{w}_{i;0,0}+\sum^{6}_{j=1}\bigg(\int_D\boldsymbol{w}_{i;0,0}\cdot \boldsymbol{\xi}_j\mathrm{d}\boldsymbol{x}\bigg)\boldsymbol{\xi}_j=\boldsymbol{\xi}_i&& \text{in } D,\\
&\partial_{\boldsymbol{\nu}}\boldsymbol{w}_{i;0,0}=0&& \text{on }\partial D.
\end{aligned}\right.
\end{align*}
Thus, $\boldsymbol{w}_{i;0,0}=\boldsymbol{\xi}_i$. For $l=1$ and $k=0$, we get
\begin{align*}
\left\{ \begin{aligned}
&-\mathcal{L}_{\lambda,\mu}\boldsymbol{w}_{i;1,0}+\sum^{6}_{j=1}\bigg(\int_D\boldsymbol{w}_{i;1,0}\cdot \boldsymbol{\xi}_j\mathrm{d}\boldsymbol{x}\bigg)\boldsymbol{\xi}_j=0&& \text{in } D,\\
&\partial_{\boldsymbol{\nu}}\boldsymbol{w}_{i;1,0}=0&& \text{on }\partial D,
\end{aligned}\right.
\end{align*}
which implies $\boldsymbol{w}_{i;1,0}=0$. By following an inductive procedure, we deduce
\begin{align*}
\boldsymbol{w}_{i;2l,0}=\rho^{l}\tau^{2l}\chi_D\boldsymbol{\xi}_i,\quad \boldsymbol{w}_{i;2l+1,0}=0,\quad \forall\, l=0,1,\cdots.
\end{align*}
Observe that for $l=0$ and $k=1$,
\begin{align*}
\left\{ \begin{aligned}
&-\mathcal{L}_{\lambda,\mu}\boldsymbol{w}_{i;0,1}+\sum^{6}_{j=1}\bigg(\int_D\boldsymbol{w}_{i;0,1}\cdot \boldsymbol{\xi}_j\mathrm{d}\boldsymbol{x}\bigg)\boldsymbol{\xi}_j=0&& \text{in } D,\\
&\partial_{\boldsymbol{\nu}}\boldsymbol{w}_{i;0,1}=\boldsymbol{\mathcal{M}}[\boldsymbol{\xi}_i]&& \text{on }\partial D.
\end{aligned}\right.
\end{align*}
Thus, $\boldsymbol{w}_{i;0,1}$ satisfies, for all $\boldsymbol{v}\in H^1(D)^3$,
\begin{align*}
&\lambda\int_{D}(\nabla\cdot\boldsymbol{w}_{i;0,1})(\nabla\cdot\overline{\boldsymbol{v}})\mathrm{d}\boldsymbol{x}
+\frac{\mu}{2}\int_{D}(\nabla\boldsymbol{w}_{i;0,1}
+\nabla\boldsymbol{w}_{i;0,1}^\top):(\nabla\overline{\boldsymbol{v}}+\overline{\boldsymbol{v}}^\top)
\mathrm{d}\boldsymbol{x}\nonumber\\
&\quad +\sum^6_{j=1}\int_D\boldsymbol{w}_{i;0,1}
\cdot\boldsymbol{\xi}_j\mathrm{d}\boldsymbol{x}\int_D\boldsymbol{\xi}_j\cdot\overline{\boldsymbol{v}}\mathrm{d}\boldsymbol{x}-\int_{\partial D}\boldsymbol{\mathcal{M}}[\boldsymbol{\xi}_i]\cdot\overline{\boldsymbol{v}}\mathrm{d}\sigma(\boldsymbol{x})=0.
\end{align*}
Setting $\boldsymbol{v}=\boldsymbol{\xi}_k$ and integrating by parts yields
\begin{align}\label{eq:w01}
\int_{ D}\boldsymbol{w}_{i;0,1}\cdot\boldsymbol{\xi}_k\mathrm{d}\boldsymbol{x}=-\mathscr{Q}_{ik},\quad k=1,\cdots,6,
\end{align}
which leads to
\begin{align*} \boldsymbol{w}_{i;0,1}=-\chi_D\sum_{k=1}^6\mathscr{Q}_{ik}\boldsymbol{\xi}_k+\widetilde{\boldsymbol{w}}_{i;0,1},
\end{align*}
where $\widetilde{\boldsymbol{w}}_{j;0,1}$ is the unique solution to \eqref{wj01}.

For $p=1$ and $k=1$,  the function $\boldsymbol{w}_{i;1,1}$ satisfies
\begin{align*}
\left\{ \begin{aligned}
&-\mathcal{L}_{\lambda,\mu}\boldsymbol{w}_{i;1,1}+\sum^{6}_{j=1}\bigg(\int_D\boldsymbol{w}_{i;1,1}\cdot \boldsymbol{\xi}_j\mathrm{d}\boldsymbol{x}\bigg)\boldsymbol{\xi}_j=0&&\text{in } D,\\
&\partial_{\boldsymbol{\nu}} \boldsymbol{w}_{i;1,1}=\sqrt{\rho}\boldsymbol{\mathcal{M}}_{1}[\boldsymbol{\xi}_i]&& \text{on }\partial D.
\end{aligned}\right.
\end{align*}
Using \eqref{eq:M1} and following a process similar to the proof of \eqref{eq:w01}, we obtain
\begin{align*}
\int_{ D}\boldsymbol{w}_{i;1,1}\cdot\boldsymbol{\xi}_k\mathrm{d}\boldsymbol{x}=\mathrm{i}\sqrt{\rho}\gamma\sum_{n=1}^3c_{ni}c_{nk},\quad k=1,\cdots,6.
\end{align*}
As a result,
\begin{align*}
\boldsymbol{w}_{i;1,1}=\text{i}\sqrt{\rho}\gamma\chi_D\sum_{k=1}^6\sum^{3}_{n=1}c_{ni}c_{nk}\boldsymbol{\xi}_k+
\text{i}\sqrt{\rho}\gamma\widetilde{\boldsymbol{w}}_{i;1,1},
\end{align*}
where $\widetilde{\boldsymbol{w}}_{i;1,1}$ is the unique solution to \eqref{E:wj11}. Thus, from the expressions of  $\boldsymbol{w}_{i;2l,0},$ $\boldsymbol{w}_{i;2l+1,0},$ $\boldsymbol{w}_{i;0,1}$, and $\boldsymbol{w}_{i;1,1}$?, we conclude that the asymptotic expansion \eqref{wjasy} holds.
\end{proof}

From \eqref{eq:b} and \eqref{wjasy}, we obtain
\begin{align}\label{I-B}
\boldsymbol{\mathscr{I}}-\boldsymbol{\mathscr{A}}(\omega,\delta)=-\rho\tau^2\omega^2\boldsymbol{\mathscr{I}}+\delta\boldsymbol{\mathscr{Q}}-\mathrm{i}
\omega\delta\sqrt{\rho}\gamma\boldsymbol{\mathscr{R}}+\mathcal{O}((\omega^2+\delta)^2),
\end{align}
where $\boldsymbol{\mathscr{R}}=(\mathscr{R}_{ij})^6_{i,j=1}$ is a real-valued  symmetric matrix, with its entries $\mathscr{R}_{ij}$ defined by \eqref{eq:pij}.

\subsection{Asymptotic expansions of subwavelength resonances}

In this subsection, we provide the asymptotic expressions for the resonant frequencies associated with the interior elastic problem \eqref{int:solu}, specifically in the regime where subwavelength resonances occur.

Introducing the new variable $\eta:=\delta^{\frac{1}{2}}$ and rescaling $\omega:=\eta \widehat{\omega}$, we obtain
\begin{align*}
\boldsymbol{\mathscr{I}}-\boldsymbol{\mathscr{A}}(\eta\widehat{\omega},\eta^2)=
-\eta^2\widehat{\omega}^2\rho\tau^2\boldsymbol{\mathscr{I}}+\eta^2\boldsymbol{\mathscr{Q}}-
\mathrm{i}\eta^3\widehat{\omega}\sqrt{\rho}\gamma\boldsymbol{\mathscr{R}}+
\mathcal{O}((\widehat{\omega}^2\eta^2+\eta^2)^2).
\end{align*}
Define the matrix-valued function
\begin{align}\label{A}
\widehat{\boldsymbol{\mathscr{A}}}(\widehat{\omega},\eta):&=\eta^{-2}(\boldsymbol{\mathscr{I}}
-\boldsymbol{\mathscr{A}}(\eta\widehat{\omega},\eta^2))\nonumber\\
&=\widehat{\boldsymbol{\mathscr{A}}}_0(\widehat{\omega})
+\eta\widehat{\boldsymbol{\mathscr{A}}}_1(\widehat{\omega})+\mathcal{O}(\eta^2),
\end{align}
where
\begin{align*}
\widehat{\boldsymbol{\mathscr{A}}}_0(\widehat{\omega})=
-\rho\tau^2\widehat{\omega}^2\boldsymbol{\mathscr{I}}+\boldsymbol{\mathscr{Q}},\quad
\widehat{\boldsymbol{\mathscr{A}}}_1(\widehat{\omega})=
-\mathrm{i}\widehat{\omega}\sqrt{\rho}\gamma
\boldsymbol{\mathscr{R}}.
\end{align*}
Then $\widehat{\boldsymbol{\mathscr{A}}}(\widehat{\omega},\cdot)$ admits an analytic continuation in a neighborhood of $\eta=0$. Observe that $\eta\widehat{\omega}$ is a characteristic value of $\boldsymbol{\mathscr{I}}-\boldsymbol{\mathscr{A}}(\eta\widehat{\omega},\eta^2)$ provided that $\widehat{\omega}$
is a characteristic value of $\widehat{\boldsymbol{\mathscr{A}}}(\widehat{\omega},\eta)$.

Since $\boldsymbol{\mathscr{Q}}$ is real-valued symmetric (cf. Lemma \ref{Q:symmetric}),
we get  the leading  term of the asymptotic expansion
\begin{align*}
\pm\widehat{\omega}_{i;0}=\pm\sqrt{\frac{\lambda_i}{\rho\tau^2}},\quad i=1,\cdots,6.
\end{align*}
Furthermore, by Kato's perturbation theory \cite{Kato1995perturbation}, the eigenvectors $\boldsymbol{\mathscr{V}}_{i}(\eta)=(\mathscr{V}_{ij}(\eta))_{j=1}^6$of $\widehat{\boldsymbol{\mathscr{A}}}(\widehat{\omega},\eta)$  can be represented as
\begin{align}\label{eq:eigenvector}
\boldsymbol{\mathscr{V}}_{i}(\eta)=\boldsymbol{\mathscr{V}}_{i}+\eta\boldsymbol{\mathscr{V}}_{i;1}+\mathcal{O}(\eta^2),\quad i=1,\cdots,6,
\end{align}
where $\boldsymbol{\mathscr{V}}_i$ are defined by \eqref{eigen:v}.

Next, we express the eigenvalues of $\boldsymbol{\mathscr{Q}}$ as $\lambda^{\mathrm{w}}_{i}$, for $i=1,\cdots,\kappa$,  with $\kappa\in\mathbb{Z}^+$ and $\kappa\leq6$, where the eigenvalues are counted without regard to multiplicities. Denote by $m_i$ the algebraic multiplicity of $\lambda^{\mathrm{w}}_{i}$ so that $\sum^{\kappa}_{i=1}m_i=6$. It is clear to note that the geometric multiplicity of $\lambda^{\mathrm{w}}_{i}$ ?is equal to its algebraic multiplicity. Consequently, the corresponding eigenfrequencies are given by
\begin{align*}
\pm\widehat{\omega}^{\mathrm{w}}_{i;0}=\pm\sqrt{\frac{\lambda^{\mathrm{w}}_i}{\rho\tau^2}},\quad i=1,\cdots,\kappa.
\end{align*}
Without loss of generality, we focus on the eigenfrequency  $\widehat{\omega}^{\mathrm{w}}_{i;0}$ for a fixed $1\leq i\leq \kappa$.

For a small enough $\gamma_0>0$, define the  $\gamma_0$-neighborhood of $\widehat{\omega}^{\mathrm{w}}_{i;0}$ as
\begin{align*}
B(\widehat{\omega}^{\mathrm{w}}_{i;0};\gamma_0)
:=\left\{\widehat{\omega}\in\mathbb{C}:|\widehat{\omega}-(\widehat{\omega}^{\mathrm{w}}_{i;0}+0\mathrm{i})|<\gamma_0\right\}.
\end{align*}
Let $M(\widehat{\omega}^{\mathrm{w}}_{i;0},\widehat{\boldsymbol{\mathscr{A}}}_0(\cdot))$ denote the algebraic multiplicity of $\widehat{\omega}^{\mathrm{w}}_{i;0}$ with respect to the boundary $\partial B(\widehat{\omega}^{\mathrm{w}}_{i;0};\gamma_0)$ (cf. \cite{Lay}), defined as
\begin{align*}
M(\widehat{\omega}^{\mathrm{w}}_{i;0},\widehat{\boldsymbol{\mathscr{A}}}_0(\cdot)):=\frac{1}{2\pi \mathrm{i}} \mathrm{tr}\int_{\partial B(\widehat{\omega}^{\mathrm{w}}_{i;0};\gamma_0)}\widehat{\boldsymbol{\mathscr{A}}}_0(\widehat{\omega})^{-1}
\frac{\mathrm{d}}{\mathrm{d}\widehat{\omega}}\widehat{\boldsymbol{\mathscr{A}}}_0(\widehat{\omega})\mathrm{d}\widehat{\omega}.
\end{align*}
Moreover, for a fixed small $\gamma_0$, we have that for any $\widehat{\omega}\in\partial B(\widehat{\omega}^{\mathrm{w}}_{i;0};\gamma_0)$,
\begin{align*}
\|(\widehat{\boldsymbol{\mathscr{A}}}(\widehat{\omega},\eta)
-\widehat{\boldsymbol{\mathscr{A}}}_0(\widehat{\omega}))\widehat{\boldsymbol{\mathscr{A}}}_0(\widehat{\omega})^{-1}\|<1
\end{align*}
as $\eta\to 0$. It follows from the generalized Rouch\'{e} theorem that
\begin{align*}
M(\partial B(\widehat{\omega}^{\mathrm{w}}_{i;0};\gamma_0),\widehat{\boldsymbol{\mathscr{A}}}(\cdot,\eta))
=M(\widehat{\omega}^{\mathrm{w}}_{i;0},\widehat{\boldsymbol{\mathscr{A}}}_0(\cdot))=m_i,
\end{align*}
where $M(\partial B(\widehat{\omega}^{\mathrm{w}}_{i;0};\gamma_0),\widehat{\boldsymbol{\mathscr{A}}}(\cdot,\eta))$
represents the number of characteristic values of $\widehat{\boldsymbol{\mathscr{A}}}(\cdot,\eta)$ in $B(\widehat{\omega}^{\mathrm{w}}_{i;0};\gamma_0)$, i.e.,
\begin{align*}
M(\partial B(\widehat{\omega}^{\mathrm{w}}_{i;0};\gamma_0),\widehat{\boldsymbol{\mathscr{A}}}(\cdot,\eta)):=\frac{1}{2\pi \mathrm{i}} \mathrm{tr}\int_{\partial B(\widehat{\omega}^{\mathrm{w}}_{i;0};\gamma_0)}\widehat{\boldsymbol{\mathscr{A}}}(\widehat{\omega},\eta)^{-1}
\frac{\partial}{\partial\widehat{\omega}}\widehat{\boldsymbol{\mathscr{A}}}(\widehat{\omega},\eta)\mathrm{d}\widehat{\omega}.
\end{align*}

\begin{rema}
Based on the above analysis, the total number of characteristic values of $\widehat{\boldsymbol{\mathscr{A}}}(\cdot,\eta)$ in $B(0;\tilde{\gamma})$ is twelve, where  $B(\pm\widehat{\omega}^{\mathrm{w}}_{i;0};\gamma_0)\subset B(0;\tilde{\gamma}), i=1,\cdots,\kappa$. Furthermore, for each $i=1,\cdots,6$, the leading term of the characteristic value $\widehat{\omega}^{\pm}_i(\eta)$ of $\widehat{\boldsymbol{\mathscr{A}}}(\widehat{\omega},\eta)$, counted with multiplicity, is $\pm\widehat{\omega}_{i;0}$.
\end{rema}

Our next objective is to identify the characteristic values of $\widehat{\boldsymbol{\mathscr{A}}}(\cdot,\eta)$ near $\widehat{\omega}^{\mathrm{w}}_{i;0}$, denoted as $\{\widehat{\omega}^{+}_{i;0,j}(\eta)\}_{j=1}^{m_i}$. Specifically, we aim to calculate the leading terms of $\widehat{\omega}^+_{i;0,j}(\eta)-\widehat{\omega}^{\mathrm{w}}_{i;0}$ for $j=1,\cdots,m_i$.

For $\ell\geq1$, define the following functions:
\begin{align*}
p_{i;\ell}(\eta):=\frac{1}{2\pi \mathrm{i}} \mathrm{tr}\int_{\partial B(\widehat{\omega}^{\mathrm{w}}_{i;0};\gamma_0)}(\widehat{\omega}-\widehat{\omega}^{\mathrm{w}}_{i;0})^{\ell}\widehat{\boldsymbol{\mathscr{A}}}(\widehat{\omega},\eta)^{-1}
\frac{\partial}{\partial\widehat{\omega}}\widehat{\boldsymbol{\mathscr{A}}}(\widehat{\omega},\eta)\mathrm{d}\widehat{\omega}.
\end{align*}
From the structure of $\widehat{\boldsymbol{\mathscr{A}}}(\widehat{\omega},\eta)$, it follows that for every $\ell\geq1$, $p_{i;\ell}(\cdot)$ has an analytic continuation in $\eta$ within the set $B(0;\eta_0):=\{\eta\in\mathbb{R}:|\eta|<\eta_0\}$ for some small $\eta_0>0$. Furthermore, by \cite[Theorem 1.14]{Lay}, we have
\begin{align*}
p_{i;\ell}(\eta)=\sum^{m_i}_{j=1}z_{i;j}(\eta)^{\ell},
\end{align*}
where $z_{i;j}(\eta):=\widehat{\omega}^+_{i;0,j}(\eta)-\widehat{\omega}^{\mathrm{w}}_{i;0}$ with $z_{i;j}(0)=0,j=1,\cdots,m_i$. These functions $p_{i;\ell}(\eta)$ are also known as Newton's functions in the variables $z_{i;j}(\eta)$ (cf. \cite{Mead1992Newton}). Define
\begin{align*}
P_{i;\eta}(\theta):=\prod_{j=1}^{m_i}(\theta-z_{i;j}(\eta)).
\end{align*}
Clearly, the values $z_{i;j}(\eta),j=1,\cdots,m_i,$ are the roots of $P_{i;\eta}(\theta)$. Observe that
\begin{align*}
P_{i;\eta}(\theta)=\theta^{m_i}+\sum_{j=1}^{m_i}(-1)^{j}\widehat{e}_{i;j}(\eta)\theta^{m_i-j},
\end{align*}
where $\widehat{e}_{i;j}(\eta)$ are elementary symmetric polynomials in the variables  $z_{i;j}(\eta)$. It follows that $\widehat{e}_{i;j}(0)=0, j=1,\cdots,m_i$. Thus, if we set $\widehat{e}_{i;0}(\eta)=1$, the functions $p_{i;\ell}(\eta)$ and $\widehat{e}_{i;j}(\eta)$ satisfy Newton's identities (cf. \cite{Mead1992Newton}):
\begin{align}\label{Newton}
\sum_{j=0}^{k-1}(-1)^{j}p_{i;k-j}(\eta)\widehat{e}_{i;j}(\eta)+(-1)^kk\widehat{e}_{i;k}(\eta)=0,\quad 1\leq k\leq m_i.
\end{align}
The recurrence relation \eqref{Newton} implies that $\widehat{e}_{i;j}(\eta),j=1,\cdots,m_i,$ are analytic in $B(0;\eta_0)$. Let $r_{i;j}(\eta):=(-1)^{j}\widehat{e}_{i;j}(\eta), j=1,\cdots,m_i,$ with $r_{i;0}(\eta)=1$. Therefore, we have
\begin{align*}
P_{i;\eta}(\theta)=\sum_{j=0}^{m_i}r_{i;j}(\eta)\theta^{m_i-j},
\end{align*}
where the functions $r_{i;j}(\eta)$ are analytic in $B(0;\eta_0)$ and satisfy $r_{i;j}(0)=(-1)^{j}\widehat{e}_{i;j}(0)=0$. Thus, $P_{i;\eta}(\theta)$ is a Weierstrass polynomial (cf. \cite{suwa2007introduction}).

\begin{lemm}
For a fixed $1\leq i\leq \kappa$, define an analytic function in the domain  $(\widehat{\omega},\eta)\in B(\widehat{\omega}^{\mathrm{w}}_{i;0};\gamma_0)\times B(0;\eta_0)$ as
\begin{align}\label{det:A}
\widetilde{f}(\widehat{\omega},\eta):=\det(\widehat{\boldsymbol{\mathscr{A}}}(\widehat{\omega},\eta)).
\end{align}
Then it admits the following factorization:
\begin{align*}
\widetilde{f}(\widehat{\omega},\eta)=P_{i;\eta}(\widehat{\omega}-\widehat{\omega}^{\mathrm{w}}_{i;0})\widetilde{g}(\widehat{\omega},\eta),
\end{align*}
where $\widetilde{g}(\cdot,\cdot)$ is analytic with respect to both $\omega$ and $\eta$, and does not vanish identically.
\end{lemm}

\begin{proof}
For every $\eta\in B(0;\eta_0)$, each zero of $P_{i;\eta}(\cdot-\widehat{\omega}^{\mathrm{w}}_{i;0})$ is also a zero of $\widetilde{f}(\cdot,\eta)$. Thus, $\widetilde{g}$ is well-defined pointwisely and is analytic in both $\widehat{\omega}$ and $\eta$. Moreover, if there exists a pair of $(\widehat{\omega}_{1},\eta_1)\in B(\widehat{\omega}^{\mathrm{w}}_{i;0};\gamma_0)\times B(0;\eta_0)$ such that $\widetilde{g}(\widehat{\omega}_1,\eta_1)=0$, then the total multiplicity of the roots of $\widetilde{f}(\widehat{\omega},\cdot)$ at the point $\eta=\eta_1$  would be at least $m_i+1$. This leads to a contradiction, completing the proof.
\end{proof}

The following result addresses the decomposition property of Weierstrass polynomials. The details can be found in \cite[p. 397; Theorem 1]{Baumgartel1985analytic}.

\begin{lemm}\label{Weierstrass}
Let $P_{i;\eta}(\widehat{\omega}-\widehat{\omega}^{\mathrm{w}}_{i;0})$ be a Weierstrass polynomial. Then it has a uniquely determined decomposition into prime factors such that
\begin{align*}
P_{i;\eta}(\widehat{\omega}-\widehat{\omega}^{\mathrm{w}}_{i;0})=\prod^{b_i}_{j=1}p_{i;\eta,j}
(\widehat{\omega}-\widehat{\omega}^{\mathrm{w}}_{i;0})^{\alpha_{i;j}},
\end{align*}
where $p_{i;\eta,j}(\cdot-\widehat{\omega}^{\mathrm{w}}_{i;0})$ are irreducible Weierstrass polynomials in $\widehat{\omega}$ of degree $\beta_{i;j}$ satisfying $\sum^{b_i}_{j=1}\alpha_{i;j}\beta_{i;j}=m_i$?, and are distinct from one another.
\end{lemm}

According to Lemma \ref{Weierstrass}, solving the equation $P_{i;\eta}(\widehat{\omega}-\widehat{\omega}^{\mathrm{w}}_{i;0})=0$ is equivalent to solving the equations $p_{i;\eta,j}
(\widehat{\omega}-\widehat{\omega}^{\mathrm{w}}_{i;0})=0,j=1,\cdots,b_i$. The zeros of $p_{i;\eta,j}
(\widehat{\omega}-\widehat{\omega}^{\mathrm{w}}_{i;0})=0$ are given by a $\beta_{i;j}$-valued algebraic function, which admits a Puiseux series expansion near $\eta=0$ (cf. \cite[p.405--406; Theorems 2 and 3]{Baumgartel1985analytic}). Therefore, we arrive at the following theorem.

\begin{theo}
As $\delta\rightarrow 0^+$, the $m_i$ zeros of $\widetilde{f}(\cdot,\eta)$ defined by \eqref{det:A} near $\widehat{\omega}^{\mathrm{w}}_{i;0}$? take on $\widetilde{m}_i$ distinct values, independent of $\eta$, denoted as $\{\widehat{\omega}^{+;\widetilde{\mathrm{w}}}_{i;0,j}(\eta)\}^{\widetilde{m}_i}_{j=1}$??, where $\widetilde{m}_i\leq{m}_i$.
Moreover, for each $j=1,\cdots,b_i$, there exists a $\beta_{i;j}$-valued algebraic function $\{(\widehat{\omega}^{+;\widetilde{\mathrm{w}}}_{i;0,j}(\eta)-\widehat{\omega}^{\mathrm{w}}_{i;0})_{l_i}\}^{\beta_{i;j}}_{l_i=1}$ that satisfies the equation $p_{i;\eta,j} (\widehat{\omega}-\widehat{\omega}^{\mathrm{w}}_{i;0})=0$ and admits the following Puiseux series expansion as $\delta\rightarrow 0^+$:
\begin{align}\label{w:expansion}
(\widehat{\omega}^{+;\widetilde{\mathrm{w}}}_{i;0,j}(\eta)-\widehat{\omega}^{\mathrm{w}}_{i;0})_{k_i}
=\sum^{+\infty}_{n=0}C_{i;j,n}\big(\eta^{\frac{1}{\beta_{i;j}}}\Big(e^{\mathrm{i}\frac{2\pi}{\beta_{i;j}}}\big)^{l_i-1}\Big)^n,
\end{align}
where
\begin{align*}
l_i=1,2,\cdots,\beta_{i;j},\quad\sum^{b_i}_{j=1}\beta_{i;j}=\widetilde{m}_i,\quad k_i=l_i+\sum^{j-1}_{p=1}\beta_{i;p},\quad
j=1,\cdots,b_i.
\end{align*}
\end{theo}

For a fixed $1\leq i\leq \kappa$, denote by $\boldsymbol{\mathscr{W}}_{i;j},j=1,\cdots,m_i$, the eigenvectors corresponding to the eigenvalue $\lambda^{\mathrm{w}}_{i}$ of $\boldsymbol{\mathscr{Q}}$ (cf. \eqref{eigen:v}). Accordingly, the eigenvectors of $\widehat{\boldsymbol{\mathscr{A}}}(\widehat{\omega},\eta)$ near $\widehat{\omega}^{\mathrm{w}}_{i;0}$  can be expressed as $\boldsymbol{\mathscr{W}}_{i;j}(\eta),j=1,\cdots,m_i$, satisfying \eqref{eq:eigenvector}. Therefore, based on \eqref{A} and \eqref{eq:eigenvector}, there exists $\widetilde{m}\in[1,m_i]$ such that
\begin{align*}
\widehat{\boldsymbol{\mathscr{A}}}((\widehat{\omega}^{+;\widetilde{\mathrm{w}}}_{i;0,j}(\eta))_{k_i},\eta)
\boldsymbol{\mathscr{W}}_{i;\widetilde{m}}(\eta)=
(-\rho\tau^2\widehat{\omega}^2+\lambda^{\mathrm{w}}_i)\boldsymbol{\mathscr{W}}_{i;\widetilde{m}}(\eta)\big|_{\widehat{\omega}
=(\widehat{\omega}^{+;\widetilde{\mathrm{w}}}_{i;0,j}(\eta))_{k_i}}+\mathcal{O}(\eta).
\end{align*}

Substituting \eqref{w:expansion} into the equality above yields
\begin{align*}
C_{i;j,n}=0,\quad n=0,1,\cdots,\beta_{i;j}-1.
\end{align*}
Thus, all the characteristic values $\widehat{\omega}^{+}_{i;0,j}(\eta),j=1,\cdots,m_i$, counted with their multiplicities, of $\widehat{\boldsymbol{\mathscr{A}}}(\cdot,\eta)$ near $\widehat{\omega}^{\mathrm{w}}_{i;0}$ can be represented as
\begin{align}\label{eq:muti}
\widehat{\omega}^{+}_{i;0,j}(\eta)=\widehat{\omega}^{\mathrm{w}}_{i;0}+C_{i;j,\beta_{i;j}}\eta+\mathcal{O}(\eta^{1+\frac{1}{\beta_{i;j}}}),
\end{align}
where, for each $j=1,\cdots,m_i$, $\beta_{i;j}$ is the degree of the irreducible Weierstrass polynomial satisfied by $\widehat{\omega}^{+}_{i;0,j}(\eta)$.

The following theorem provides the asymptotic representations of the subwavelength resonance frequencies.

\begin{theo}\label{theoasy}
Let $\omega\in\mathbb{C}$, $\delta\in\mathbb{R}^+$, and  $\epsilon\in\mathbb{R}^+$. All the characteristic values $\omega^{\pm}_i(\epsilon^{\frac{1}{2}})$, counted with their multiplicities,  of $\boldsymbol{\mathscr{I}}-\boldsymbol{\mathscr{A}}(\omega,\delta)$, representing the subwavelength resonant frequencies, have the following the asymptotic expansions as $\epsilon\rightarrow 0^+$:
\begin{align}\label{full:asymptotic}
\omega^{\pm}_i(\epsilon^{\frac{1}{2}})=\pm\sqrt{\frac{\lambda_i}{\rho}}\epsilon^{\frac{1}{2}}
-\mathrm{i}\frac{\gamma\boldsymbol{\mathscr{V}}^\top_i\boldsymbol
{\mathscr{R}}\boldsymbol{\mathscr{V}}_i}{2\sqrt{\rho}}\epsilon+\mathcal{O}(\epsilon^{1+\frac{1}{2\beta_{i}}}),\quad i=1,\cdots,6,
\end{align}
with the relation
\begin{align}\label{-reson}
\omega^-_i(\epsilon^{\frac{1}{2}})=-\overline{\omega^+_i}(\epsilon^{\frac{1}{2}}),
\end{align}
where  $\lambda_i$ and $\boldsymbol{\mathscr{V}}_i$ are the eigenvalues and the corresponding eigenvectors of $\boldsymbol{\mathscr{Q}}$, respectively, $\boldsymbol{\mathscr{Q}}$ and $\boldsymbol
{\mathscr{R}}$ are defined in \eqref{matrix:QR}, $\beta_{i}$ stands for the degree of the irreducible Weierstrass polynomial satisfied by $\omega^{\pm}_i(\epsilon^{\frac{1}{2}})$, and $\gamma$ is given by \eqref{ri}.
\end{theo}

\begin{proof}
By \eqref{eq:muti}, a straightforward calculation yields
\begin{align*}
\widehat{\boldsymbol{\mathscr{A}}}(\widehat{\omega}^{+}_{i;0,j}(\eta),\eta)\boldsymbol{\mathscr{V}}_{i;j}(\eta)
&=(-\lambda^{\mathrm{w}}_{i}\boldsymbol{\mathscr{I}}+\boldsymbol{\mathscr{Q}})\boldsymbol{\mathscr{V}}_{i;j}+
\eta(-2\rho\tau^2\widehat{\omega}^{\mathrm{w}}_{i;0}C_{i;j,\beta_{i;j}}\boldsymbol{\mathscr{I}}
-\mathrm{i}\widehat{\omega}^{\mathrm{w}}_{i;0}\sqrt{\rho}\gamma\boldsymbol{\mathscr{R}})\boldsymbol{\mathscr{V}}_{i;j}\\
&\quad+\eta(-\lambda^{\mathrm{w}}_{i}\boldsymbol{\mathscr{I}}
+\boldsymbol{\mathscr{Q}})\boldsymbol{\mathscr{V}}_{i;j,1})+\mathcal{O}(\eta^{1+\frac{1}{\beta_{i;j}}})=0.
\end{align*}
Therefore, we have
\begin{equation*}
C_{i;j,\beta_{i;j}}=-\mathrm{i}\frac{\gamma\boldsymbol{\mathscr{V}}^{\top}_{i;j}
\boldsymbol{\mathscr{R}}\boldsymbol{\mathscr{V}}_{i;j}}{2\sqrt{\rho}\tau^2}.
\end{equation*}
Since $\eta=\delta^{\frac{1}{2}}$ and $\omega=\eta \widehat{\omega}$, we can derive from Remark \ref{resonance} that
\begin{align}\label{Omega:delta}
\omega^{\pm}_i(\delta^{\frac{1}{2}})=\pm\sqrt{\frac{\lambda_i}{\rho\tau^2}}\delta^{\frac{1}{2}}-\mathrm{i}\frac{\gamma\boldsymbol{\mathscr{V}}^\top_i\boldsymbol
{\mathscr{R}}\boldsymbol{\mathscr{V}}_i}{2\sqrt{\rho}\tau^2}\delta+\mathcal{O}(\delta^{1+\frac{1}{2\beta_{i}}}),\quad i=1,\cdots,6,
\end{align}
counted with their multiplicities. Thus, using \eqref{contrast} leads to the conclusion that \eqref{full:asymptotic} and \eqref{-reson} hold.
\end{proof}

\section{Asymptotic analysis of the scattered field}\label{sec:4}

The objective of this section is to establish the asymptotic expansion of the solution to the interior elastic problem and to provide the far-field patterns, which include both the transverse and longitudinal components.

\subsection{Solution to the interior problem}

The aim of this subsection is to utilize the resonant frequencies to represent the solution to the interior problem
\eqref{int:solu}. This solution is related to $\boldsymbol{w}_i(\omega,\delta), \boldsymbol{w}_f(\omega,\delta)$, and $\boldsymbol{s}(\omega,\delta)$, as defined by \eqref{wi}, \eqref{wf}, and \eqref{eq:psi}, respectively.

Without loss of generality, we can assume that the center of the elastic body $D$ is located at the origin. Consider the incidence of a time-harmonic plane wave given by
\begin{align}\label{eq:compressed}
\boldsymbol{u}^{\mathrm{in}}(\boldsymbol x)=\boldsymbol{d}e^{\mathrm{i}k_p\boldsymbol x\cdot \boldsymbol{d}},
\end{align}
which is known as the compressed plane wave, where $\boldsymbol{d}$ is the unit incident direction vector and $k_p$ is defined by \eqref{wave number}.

The following proposition shows the asymptotic expansion of $\boldsymbol{w}_f(\omega,\delta)$ with respect to $\omega$ and $\delta$.

\begin{prop}\label{prop:wf}
Let $\omega\in\mathbb{C}$ and  $\delta\in\mathbb{R}^+$. The following asymptotic representation holds for the solution $\boldsymbol{w}_{f}$ to the variational problem \eqref{wf} as $\omega\rightarrow0$ and $\delta\rightarrow 0^+$:
\begin{align}\label{wfasy}
\boldsymbol{w}_{\boldsymbol{f}}(\omega,\delta)=\delta\chi_D\sum_{k=1}^6\sum_{n=1}^3(\boldsymbol{d}\cdot \mathbf{e}_n)c_{nk}\boldsymbol{\xi}_k+\delta\widetilde{\boldsymbol{w}}_{\boldsymbol f;0,1}+\mathcal{O}(\omega\delta+\delta^2),
\end{align}
where $c_{nk}$ are defined in \eqref{eq:pij}, and $\widetilde{\boldsymbol{w}}_{\boldsymbol f;0,1}$ is the unique solution to the following elastic problem:
\begin{align}\label{v01}
\left\{ \begin{aligned}
&-\mathcal{L}_{\lambda,\mu}\widetilde{\boldsymbol{w}}_{\boldsymbol f;0,1}+\sum_{k=1}^6\sum_{n=1}^3(\boldsymbol{d}\cdot \mathbf{e}_{n})c_{nk}\boldsymbol{\xi}_k &&\text{\rm in } D,\\
&\partial_{\boldsymbol{\nu}}\widetilde{\boldsymbol{w}}_{\boldsymbol f;0,1}=-\sum_{n=1}^3(\boldsymbol{d}\cdot \mathbf{e}_{n})\boldsymbol{\mathcal{M}}[\mathbf{e}_n]&& \text{\rm on }\partial D,\\
&\int_D\widetilde{\boldsymbol{w}}_{\boldsymbol f;0,1}\cdot\boldsymbol{\xi}_k\mathrm{d}\boldsymbol{x}=0,&&k=1,\cdots,6.
\end{aligned}\right.
\end{align}
\end{prop}

\begin{proof}
Thanks to the definition of $a_{\omega,\delta}({\boldsymbol{u},\boldsymbol{v}})$ given in \eqref{a(u,v)} and the  variational problem stated in \eqref{wf}, the following strong form holds:
\begin{align}\label{wfPDE}
\left\{ \begin{aligned}
&-\mathcal{L}_{\lambda,\mu}\boldsymbol{w}_{\boldsymbol{f}}(\omega,\delta)+\sum^{6}_{i=1}\bigg(\int_D\boldsymbol{w}_{\boldsymbol{f}}(\omega,\delta)\cdot \boldsymbol{\xi}_i\mathrm{d}\boldsymbol{x}\bigg)\boldsymbol{\xi}_i-\rho\tau^2\omega^2\boldsymbol{w}_{\boldsymbol{f}}(\omega,\delta)=0&&\text{in }D,\\
&\partial_{\boldsymbol{\nu}}\boldsymbol{w}_{\boldsymbol{f}}(\omega,\delta)=\delta \boldsymbol{\mathcal{M}}^{\sqrt{\rho}\omega}[\boldsymbol{w}_{\boldsymbol{f}}(\omega,\delta)-\boldsymbol{u}^{\mathrm{in}}]+\delta\partial_{\boldsymbol{\nu}}{\boldsymbol{u}^{\mathrm{in}}}&& \text{on }\partial D.\\
\end{aligned}\right.
\end{align}
It is clear that
\begin{align}\label{eq:incidence}
\boldsymbol{u}^{\mathrm{in}}(\boldsymbol x)=\sum^{+\infty}_{l=0}\omega^l\boldsymbol{u}^{\mathrm{in}}_l(\boldsymbol x),
\end{align}
where $\boldsymbol{u}^{\mathrm{in}}_l(\boldsymbol x):=\frac{\mathrm{i}^{l}(\boldsymbol x\cdot \boldsymbol{d})^l}{l!}\big(\frac{\rho}{\lambda+2\mu}\big)^{\frac{l}{2}}\boldsymbol{d}.$  On the other hand, it follows from the Taylor expansion of $\boldsymbol{u}^{\mathrm{in}}(\cdot)$ at the origin that
\begin{align*}
\boldsymbol{u}^{\mathrm{in}}(\boldsymbol x)=\sum^{+\infty}_{l=0}\frac{1}{l!}(\boldsymbol x \cdot \nabla)^{l}\boldsymbol{u}^{\mathrm{in}}(0),
\end{align*}
where $(\boldsymbol x \cdot \nabla)^{l}:=\sum_{1\leq i_1,i_2,\cdots,i_l\leq3}x_{i_1}x_{i_2}\cdots x_{i_l}\partial^{l}_{x_{i_1}x_{i_2}\cdots x_{i_l}}$.
Consequently, we obtain
\begin{align*}
\omega^l\boldsymbol{u}^{\mathrm{in}}_l(\boldsymbol x)=\frac{1}{l!}(\boldsymbol x \cdot \nabla)^{l}\boldsymbol{u}^{\mathrm{in}}(0), \quad
\boldsymbol{u}^{\mathrm{in}}_0(\boldsymbol x)=\boldsymbol{u}^{\mathrm{in}}(0)=\boldsymbol{d}.
\end{align*}

Notice that $\boldsymbol{w}_{\boldsymbol f}(\omega,\delta)$ is an analytic function of $\omega$ and $\delta$. Hence, there are functions $\{\boldsymbol{w}_{i;l,k}\}_{l,k=0}^{+\infty}$, which the corresponding power series converges in $H^1(D)^3$:
\begin{align}\label{vpn}
\boldsymbol{w}_{f}(\omega,\delta)=\sum_{l,k=0}^{+\infty}\omega^l\delta^k\boldsymbol{w}_{\boldsymbol{f};l,k}.
\end{align}
We assume that $\boldsymbol{w}_{\boldsymbol{f};l,k}=0$ for $l,k< 0.$
Substituting \eqref{Se:DtN}, \eqref{eq:incidence}, and \eqref{vpn} into \eqref{wfPDE} yields
\begin{align*}
\left\{ \begin{aligned}
&-\mathcal{L}_{\lambda,\mu}\boldsymbol{w}_{\boldsymbol{f};l,k}+\sum^{6}_{i=1}\bigg(\int_D\boldsymbol{w}_{\boldsymbol{f};l,k}\cdot \boldsymbol{\xi}_i\mathrm{d}\boldsymbol{x}\bigg)\boldsymbol{\xi}_i-\rho\tau^2\boldsymbol{w}_{\boldsymbol{f};l-2,k}=0&& \text{in } D,\\
&\partial_{\boldsymbol{\nu}}\boldsymbol{w}_{\boldsymbol{f};l,k}=\sum^l_{m=0}\rho^{\frac{m}{2}}\boldsymbol{\mathcal{M}}_{m}[\boldsymbol{w}_{\boldsymbol{f};l-m,k-1}]
+\Big(\partial_{\boldsymbol{\nu}} \boldsymbol{u}^{\mathrm{in}}_l-\sum^l_{m=0}\rho^{\frac{m}{2}}\boldsymbol{\mathcal{M}}_{m}[\boldsymbol{u}^{\mathrm{in}}_{l-m}]\Big)\delta_{k=1}&&\text{on }\partial D.
\end{aligned}\right.
\end{align*}

If $l=k=0$, then
\begin{align*}
\left\{ \begin{aligned}
&-\mathcal{L}_{\lambda,\mu}\boldsymbol{w}_{\boldsymbol{f};0,0}+\sum^{6}_{i=1}\bigg(\int_D\boldsymbol{w}_{\boldsymbol{f};0,0}\cdot \boldsymbol{\xi}_i\mathrm{d}\boldsymbol{x}\bigg)\boldsymbol{\xi}_i=0&& \text{in } D,\\
&\partial_{\boldsymbol{\nu}}\boldsymbol{w}_{\boldsymbol{f};0,0}=0&& \text{on }\partial D,
\end{aligned}\right.
\end{align*}
which leads to $\boldsymbol{w}_{\boldsymbol{f};0,0}=0$. Consequently, by induction, we get $\boldsymbol{w}_{\boldsymbol{f};l,0}=0$ for any $l\geq0$.

For $l=0$ and $k=1$, we consider
\begin{align*}
\left\{ \begin{aligned}
&-\mathcal{L}_{\lambda,\mu}\boldsymbol{w}_{\boldsymbol{f};0,1}+\sum^{6}_{i=1}\bigg(\int_D\boldsymbol{w}_{\boldsymbol{f};0,1}\cdot \boldsymbol{\xi}_i\mathrm{d}\boldsymbol{x}\bigg)\boldsymbol{\xi}_i=0&& \text{in } D,\\
&\partial_{\boldsymbol{\nu}}\boldsymbol{w}_{\boldsymbol{f};0,1}=-\sum_{n=1}^3(\boldsymbol{d}\cdot \mathbf{e}_n)\boldsymbol{\mathcal{M}}[\mathbf{e}_n]&& \text{on }\partial D.
\end{aligned}\right.
\end{align*}
Using a similar procedure as that in the proof of \eqref{eq:w01}, we have
\begin{align*}
\int_{ D}\boldsymbol{w}_{\boldsymbol{f};0,1}\cdot\boldsymbol{\xi}_k\mathrm{d}\boldsymbol{x}=\int_{\partial D}\partial_{\boldsymbol{\nu}}\boldsymbol{w}_{\boldsymbol{f};0,1}\cdot\boldsymbol{\xi}_k\mathrm{d}\sigma(\boldsymbol{x})=\sum^3_{n=1}(\boldsymbol{d}\cdot \mathbf{e}_n)c_{nk}.
\end{align*}
Therefore, we can deduce that
\begin{align*}
\boldsymbol{w}_{\boldsymbol{f};0,1}=\sum_{k=1}^6\sum_{n=1}^3(\boldsymbol{d}\cdot \mathbf{e}_n)c_{nk}\boldsymbol{\xi}_k+\widetilde{\boldsymbol{w}}_{\boldsymbol{f};0,1},
\end{align*}
where $\widetilde{\boldsymbol{w}}_{\boldsymbol{f};0,1}$ is the unique solution to \eqref{v01}. As a result, we obtain \eqref{wfasy}, which completes the proof.
\end{proof}

\begin{rema}
Similar conclusions to those in Proposition \ref{prop:wf} hold if the incident waves take the following forms:
\begin{enumerate}

\item[(i)] Consider the shear plane wave $\boldsymbol{u}^{\mathrm{in}}(\boldsymbol{x})=\boldsymbol{d}^{\bot}e^{\mathrm{i}k_s\boldsymbol x\cdot \boldsymbol{d}}$, where $\boldsymbol{d}^{\bot}$ is the unit polarization vector, satisfying $\boldsymbol{d}\cdot \boldsymbol{d}^\bot=0$, and $k_s$? is defined by \eqref{wave number}.

\item[(ii)] Consider a linear combination of a shear plane wave and a compression plane wave.
\end{enumerate}
\end{rema}

Let $\boldsymbol{\mathscr{C}}$ be a $3\times6$ matrix with its $(i,j)$-th entry defined as $\mathscr{C}_{ij}:=c_{ij}$, where $c_{ij}$ are given by \eqref{eq:pij}.  Based on Proposition \ref{prop:wf}, we establish the asymptotic expansion of the $i$-th component $g_i(\omega,\delta)$ of $\boldsymbol{g}(\omega,\delta)$. It follows from \eqref{eq:g}
 and \eqref{wfasy} that
\begin{align*}
g_{i}(\omega,\delta)=\delta\sum_{n=1}^3(\boldsymbol{d}\cdot\mathbf{e}_n)c_{ni}+\mathcal{O}(\omega\delta+\delta^2),
\end{align*}
which leads to
\begin{align}\label{Fasy}
\boldsymbol{g}(\omega,\delta)=\delta\boldsymbol{\mathscr{C}}^\top\boldsymbol{d}+\mathcal{O}(\omega\delta+\delta^2).
\end{align}

For subsequent analysis, we make the following assumption.

\begin{assu}
For every resonant frequency $\omega^{\pm}_i(\delta^{\frac{1}{2}})$ or $\omega^{\pm}_i(\epsilon^{\frac{1}{2}})$, we assume that the coefficient corresponding to the imaginary part of order $\mathcal{O}(\delta)$ does not vanish. Specifically, we require that
\begin{align} \label{a:imaginary}
\boldsymbol{\mathscr{V}}^\top_i\boldsymbol
{\mathscr{R}}\boldsymbol{\mathscr{V}}_i>0,\quad i=1,\cdots,6.
\end{align}
\end{assu}
In fact, resonances only occur when $\omega \in \mathbb{C}$. Since $\boldsymbol{\mathscr{R}}$ is  positive semi-definite (cf. Lemma \ref{Q:symmetric}), it is possible that $\boldsymbol{\mathscr{V}}^\top_i\boldsymbol
{\mathscr{R}}\boldsymbol{\mathscr{V}}_i=0$ for some $i$ in certain cases. This indicates that the imaginary parts exhibit higher order behavior in $\delta$ or $\epsilon$.

\begin{lemm}\label{le:enhan}
Let $\omega\in\mathbb{R}$ and $\delta\in\mathbb{R}^+$ satisfy $\omega=\mathcal{O}(\delta^{\frac{1}{2}})$. Under the assumption \eqref{a:imaginary}, we have for $i=1,\cdots,6$,
\begin{align}\label{error}
\frac{\delta}{(\omega-\omega^+_i(\delta^{\frac{1}{2}}))(\omega-\omega^-_i(\delta^{\frac{1}{2}}))}=\frac{\delta(1+\mathcal{O}(\delta^{\frac{1}{2\beta_i}}))}{\omega^2-\frac{\lambda_i}{\rho\tau^2}\delta+\mathrm{i}\frac{\gamma\boldsymbol{\mathscr{V}}^\top_i\boldsymbol
{\mathscr{R}}\boldsymbol{\mathscr{V}}_i}{\sqrt{\rho}\tau^2}\omega\delta}=\mathcal{O}(\delta^{-\frac12})
\end{align}
as $\delta\rightarrow0^{+}$, where the definitions of $\beta_i$ are given in Theorem \ref{theoasy}.
\end{lemm}

\begin{proof}
Observe that
\begin{align*}
(\omega-\omega^+_i(\delta^{\frac{1}{2}}))(\omega-\omega^-_i(\delta^{\frac12}))\stackrel{\eqref{-reson}}{=}\omega^2-|\omega^{\pm}_i(\delta^{\frac{1}{2}})|^2
-2\mathrm{i}\omega\Im(\omega^{\pm}_i(\delta^{\frac{1}{2}})).
\end{align*}
It follows from \eqref{Omega:delta} that
\begin{align*}
-|\omega^{\pm}_i(\delta^{\frac{1}{2}})|^2&=-\frac{\lambda_i}{\rho\tau^2}\delta+\mathcal{O}(\delta^{\frac{3}{2}+\frac{1}{2\beta_i}}),\\
-2\mathrm{i}\omega\Im(\omega^{\pm}_i(\delta^{\frac{1}{2}}))&=\mathrm{i}\frac{\gamma\boldsymbol{\mathscr{V}}^\top_i\boldsymbol
{\mathscr{R}}\boldsymbol{\mathscr{V}}_i}{\sqrt{\rho}\tau^2}\omega\delta+\mathcal{O}(\omega\delta^{1+\frac{1}{2\beta_i}}).
\end{align*}
Furthermore, using \cite[Lemma 4]{Fepponmodal}, we obtain
\begin{align*}
\bigg|\omega^2-\frac{\lambda_i}{\rho\tau^2}\delta+\mathrm{i}\frac{\gamma\boldsymbol{\mathscr{V}}^\top_i\boldsymbol
{\mathscr{R}}\boldsymbol{\mathscr{V}}_i}{\sqrt{\rho}\tau^2}\omega\delta\bigg|\geq\frac{\gamma\boldsymbol{\mathscr{V}}^\top_i\boldsymbol {\mathscr{R}}\boldsymbol{\mathscr{V}}_i\delta}{\sqrt{\rho}\tau^2}\sqrt{\frac{\lambda_i\delta}{\rho\tau^2}-\bigg(\frac{\gamma\boldsymbol{\mathscr{V}}^\top_i\boldsymbol
{\mathscr{R}}\boldsymbol{\mathscr{V}}_i\delta}{2\sqrt{\rho}\tau^2}\bigg)^2},
\end{align*}
where the equal sign is taken when $\omega^2=\frac{\lambda_i\delta}{\rho\tau^2}-2\Big(\frac{\gamma\boldsymbol{\mathscr{V}}^\top_i\boldsymbol
{\mathscr{R}}\boldsymbol{\mathscr{V}}_i\delta}{2\sqrt{\rho}\tau^2}\Big)^2$. Consequently, for $\omega\in\mathbb{R}$ with $\omega=\mathcal{O}(\delta^{\frac{1}{2}})$, we have
\begin{align*}
\frac{\delta}{(\omega-\omega^+_i(\delta^{\frac{1}{2}}))(\omega-\omega^-_i(\delta^{\frac{1}{2}}))}
&=\frac{\delta}{\Big(\omega^2-\frac{\lambda_i}{\rho\tau^2}\delta+\mathrm{i}\frac{\gamma\boldsymbol{\mathscr{V}}^\top_i\boldsymbol {\mathscr{R}}\boldsymbol{\mathscr{V}}_i}{\sqrt{\rho}\tau^2}\omega\delta\Big)\Big(1+\mathcal{O}(\delta^{\frac{1}{2\beta_i}})\Big)}\nonumber\\
&=\frac{\delta(1+\mathcal{O}(\delta^{\frac{1}{2\beta_i}}))}{\omega^2
-\frac{\lambda_i}{\rho\tau^2}\delta+\mathrm{i}\frac{\gamma\boldsymbol{\mathscr{V}}^\top_i\boldsymbol
{\mathscr{R}}\boldsymbol{\mathscr{V}}_i}{\sqrt{\rho}\tau^2}\omega\delta},
\end{align*}
which completes the proof.
\end{proof}

The following result provides the asymptotic expansion of $\boldsymbol{s}(\omega,\delta)$ with respect to $\omega$ and $\delta$.

\begin{prop}\label{prop:si}
Let $\omega\in\mathbb{R}$ and  $\delta\in\mathbb{R}^+$ satisfy $\omega=\mathcal{O}(\delta^{\frac{1}{2}})$. Under the assumption \eqref{a:imaginary}, for every $1\leq i\leq 6$, the $i$-th component $s_i(\omega,\delta)$ of $\boldsymbol{s}(\omega,\delta)$  has the following asymptotic expansion as $\delta\rightarrow 0^+$:
\begin{align}\label{first coefficient}
s_i(\omega,\delta)=\sum^6_{j=1}\frac{\delta \mathscr{V}_{ji}(\boldsymbol{\mathscr{V}}^{\top}\boldsymbol{\mathscr{C}}^\top\boldsymbol{d})_j(1+\mathcal{O}(\delta^{\frac{1}{2\beta_{j}}}))}
{-\rho\tau^2\omega^2+\lambda_j\delta-\mathrm{i}\gamma\sqrt{\rho}\boldsymbol{\mathscr{V}}^\top_j\boldsymbol{\mathscr{R}}
\boldsymbol{\mathscr{V}}_j\omega\delta}+\mathcal{O}(1),
\end{align}
where $(\boldsymbol{\mathscr{V}}^{\top}\boldsymbol{\mathscr{C}}^\top\boldsymbol{d})_j$ denotes the $j$-th component of the vector $\boldsymbol{\mathscr{V}}^{\top}\boldsymbol{\mathscr{C}}^\top\boldsymbol{d}$.
\end{prop}

\begin{proof}
We define an invertible $6\times6$  matrix $\boldsymbol{\mathscr{V}}(\delta^{\frac{1}{2}})$ as
\begin{align*}
\boldsymbol{\mathscr{V}}(\delta^{\frac{1}{2}})
:=(\boldsymbol{\mathscr{V}}_{1}(\delta^{\frac{1}{2}}),\cdots,\boldsymbol{\mathscr{V}}_6(\delta^{\frac{1}{2}})),
\end{align*}
where $\boldsymbol{\mathscr{V}}_{i}(\delta^{\frac{1}{2}})$ are given in \eqref{eq:eigenvector}. We introduce the change of variable
\begin{align*}
\widetilde{\boldsymbol{s}}(\omega,\delta):=\boldsymbol{\mathscr{V}}(\delta^{\frac{1}{2}})^{\top}\boldsymbol{s}(\omega,\delta),
\end{align*}
which leads to  $\boldsymbol{s}(\omega,\delta)=\boldsymbol{\mathscr{V}}(\delta^{\frac{1}{2}})\widetilde{\boldsymbol{s}}(\omega,\delta)$.

It follows from \eqref{matrix}, \eqref{I-B}, and \eqref{Fasy} that
\begin{align*}
\mathrm{diag}\big(-\rho\tau^2(\omega-\omega^+_i(\delta^{\frac{1}{2}}))
(\omega-\omega^-_i(\delta^{\frac{1}{2}}))\big)^6_{i=1}\widetilde{\boldsymbol{s}}(\omega,\delta)
&=\boldsymbol{\mathscr{V}}(\delta^{\frac{1}{2}})^{\top}\boldsymbol{g}(\omega,\delta)\\
&=\delta\boldsymbol{\mathscr{V}}^{\top}\boldsymbol{\mathscr{C}}^\top\boldsymbol{d}+\mathcal{O}(\delta^{\frac32})+
\mathcal{O}(\omega\delta).
\end{align*}
Then, for $\omega\in\mathbb{R}$, we have
\begin{align*}
\widetilde{\boldsymbol{s}}(\omega,\delta)=
\mathrm{diag}\Bigg(\frac{\delta}{-\rho\tau^2(\omega-\omega^+_i(\delta^{\frac{1}{2}}))
(\omega-\omega^-_i(\delta^{\frac{1}{2}}))}\Bigg)^6_{i=1}\big(\boldsymbol{\mathscr{V}}^{\top}\boldsymbol{\mathscr{C}}^\top\boldsymbol{d}
+\mathcal{O}(\delta^{\frac12})+
\mathcal{O}(\omega)\big).
\end{align*}
Combining the above equation with \eqref{error}, we deduce \eqref{first coefficient} for $\omega\in\mathbb{R}$ with $\omega=\mathcal{O}(\delta^{\frac{1}{2}})$.
\end{proof}

Based on Propositions \ref{wi:expansion}, \ref{prop:wf}, and \ref{prop:si}, we derive the asymptotic expansion for the solution to the interior problem, as stated in the following theorem.

\begin{theo}\label{inter u asy}
Let $\omega\in\mathbb{R}$ and  $\delta\in\mathbb{R}^+$ satisfy $\omega=\mathcal{O}(\delta^{\frac{1}{2}})$. Under the assumption \eqref{a:imaginary}, the solution $\boldsymbol{u}(\omega,\delta)$ to the interior problem \eqref{int:solu}   admits the following asymptotic expansion as $\delta\rightarrow 0^+$:
\begin{align}\label{u:asymptotic}
\boldsymbol{u}(\omega,\delta)=\chi_D\sum^6_{i,j=1}\frac{\delta \mathscr{V}_{ji}(\boldsymbol{\mathscr{V}}^{\top}\boldsymbol{\mathscr{C}}^\top\boldsymbol{d})_j(1+\mathcal{O}(\delta^{\frac{1}{2\beta_{j}}}))}
{-\rho\tau^2\omega^2+\lambda_j\delta-\mathrm{i}\gamma\sqrt{\rho}\boldsymbol{\mathscr{V}}^\top_j\boldsymbol{\mathscr{R}}
\boldsymbol{\mathscr{V}}_j\omega\delta}\boldsymbol{\xi}_i+\mathcal{O}(1),
\end{align}
where $\beta_j$ are given in Theorem \ref{theoasy}. Moreover, the resonance enhancement coefficients of the field $\boldsymbol{u}(\omega,\delta)$ in the domain $D$ are of order $\mathcal{O}(\delta^{-\frac{1}{2}})$.
\end{theo}

\begin{proof}
From \eqref{solu:express}, \eqref{wjasy}, \eqref{wfasy}, \eqref{first coefficient}, and \eqref{error}, it follows that
\begin{align*}
\boldsymbol{u}(\omega,\delta)
&=\sum^6_{i,j=1}\frac{\delta \mathscr{V}_{ji}(\boldsymbol{\mathscr{V}}^{\top}\boldsymbol{\mathscr{C}}^\top\boldsymbol{d})_j(1+\mathcal{O}(\delta^{\frac{1}{2\beta_{j}}}))}
{-\rho\tau^2\omega^2+\lambda_j\delta-\mathrm{i}\gamma\sqrt{\rho}\boldsymbol{\mathscr{V}}^\top_j\boldsymbol{\mathscr{R}}
\boldsymbol{\mathscr{V}}_j\omega\delta}\left(\chi_D\boldsymbol{\xi}_i+\mathcal{O}(\delta)\right)+\mathcal{O}(\delta)\\
&=\chi_D\sum^6_{i,j=1}\frac{\delta \mathscr{V}_{ji}(\boldsymbol{\mathscr{V}}^{\top}\boldsymbol{\mathscr{C}}^\top\boldsymbol{d})_j(1+\mathcal{O}(\delta^{\frac{1}{2\beta_{j}}}))}
{-\rho\tau^2\omega^2+\lambda_j\delta-\mathrm{i}\gamma\sqrt{\rho}\boldsymbol{\mathscr{V}}^\top_j\boldsymbol{\mathscr{R}}
\boldsymbol{\mathscr{V}}_j\omega\delta}\boldsymbol{\xi}_i+\mathcal{O}(1),
\end{align*}
which demonstrates that the enhancement coefficients are of order $\mathcal{O}(\delta^{-\frac{1}{2}})$.
\end{proof}

\begin{rema}
In fact, using Lemma \ref{le:enhan}, the enhancement coefficients of the resonances associated with the field $\boldsymbol{u}(\omega,\delta)$ in the domain $D$ are governed by the imaginary parts of the resonant frequencies.
\end{rema}

\subsection{Far-field patterns}

In this subsection, we present the representation of the solution to the exterior elastic problem in the far-field, based on \eqref{ext:solu} and \eqref{u:asymptotic}.

We begin by recalling an important property of the Kupradze matrix $\boldsymbol{\Gamma}^{k}(\boldsymbol x-\boldsymbol y)$, as discussed in \cite{DURGAextraction}.

\begin{lemm}
As $|\boldsymbol x|\rightarrow +\infty$ and $k\rightarrow0$,  the Kupradze matrix $\boldsymbol{\Gamma}^{k}(\boldsymbol x-\boldsymbol y)$ for any $\boldsymbol{y}\in\partial D$ has the following asymptotic behavior:
\begin{align}\label{fun solu far}
\boldsymbol\Gamma^{k}(\boldsymbol x-\boldsymbol y)&=-\frac{e^{\mathrm{i}\frac{k}{\sqrt{\mu}}|\boldsymbol x|}}{4\pi\mu|\boldsymbol x|}(\mathbf{I}-\widehat{\boldsymbol x}\widehat{\boldsymbol x}^\top)e^{-\mathrm{i} \frac{k}{\sqrt{\mu}}\widehat{\boldsymbol x}\cdot \boldsymbol y}-\frac{e^{\mathrm{i}\frac{k}{\sqrt{\lambda+2\mu}}|\boldsymbol x|}}{4\pi(\lambda+2\mu)|\boldsymbol x|}\widehat{\boldsymbol x}\widehat{\boldsymbol x}^\top e^{-\mathrm{i} \frac{k}{\sqrt{\lambda+2\mu}}\widehat{\boldsymbol x}\cdot \boldsymbol y}+\mathcal{O}(|\boldsymbol x|^{-2}).
\end{align}
\end{lemm}

Based on the far-field pattern of the Kupradze matrix, the following result addresses the asymptotic expansions of the far-field patterns associated with the elastic problem \eqref{Lame} (cf. \cite{Carlos2002on}).

\begin{theo}\label{far}
Let $\omega\in\mathbb{R}$ and  $\delta\in\mathbb{R}^+$ satisfy $\omega=\mathcal{O}(\delta^{\frac{1}{2}})$. Then the solution to \eqref{Lame} in the far-field admits the following form as $\delta\rightarrow0^+$:
\begin{align}\label{far solution}
\boldsymbol{u}^{\mathrm{ex}}(\boldsymbol x)-\boldsymbol{u}^{\mathrm{in}}(\boldsymbol x)&=-\frac{e^{\mathrm{i}\frac{\sqrt{\rho}\omega}{\sqrt{\mu}}|\boldsymbol x|}}{|\boldsymbol x|}\boldsymbol{u}_{s,\infty}(\widehat{\boldsymbol x})-\frac{e^{\mathrm{i}\frac{\sqrt{\rho}\omega}{\sqrt{\lambda+2\mu}}|\boldsymbol x|}}{|\boldsymbol x|}\boldsymbol{u}_{p,\infty}(\widehat{\boldsymbol x})+\mathcal{O}(|\boldsymbol x|^{-1})+\mathcal{O}(\delta^{-\frac{1}{2}}|\boldsymbol x|^{-2}),
\end{align}
where
\begin{align*}
\boldsymbol{u}_{s,\infty}(\widehat{\boldsymbol x}):=\frac{1}{4\pi\mu}(\mathbf{I}-\widehat{\boldsymbol x}\widehat{\boldsymbol x}^\top)\sum^6_{i,j=1}\frac{\delta \mathscr{V}_{ji}(\boldsymbol{\mathscr{V}}^{\top}\boldsymbol{\mathscr{C}}^\top\boldsymbol{d})_j}
{-\rho\tau^2\omega^2+\lambda_j\delta-\mathrm{i}\gamma\sqrt{\rho}\boldsymbol{\mathscr{V}}^\top_j\boldsymbol{\mathscr{R}}
\boldsymbol{\mathscr{V}}_j\omega\delta}\\
\times \int_{\partial D}\boldsymbol{\mathcal{S}}^{-1}_D[\boldsymbol{\xi}_i](1+\mathcal{O}(\delta^{\frac{1}{2\beta_{j}}}))\mathrm{d}\sigma(\boldsymbol y),\\
\boldsymbol{u}_{p,\infty}(\widehat{\boldsymbol x}):=\frac{1}{4\pi(\lambda+2\mu)}\widehat{\boldsymbol x}\widehat{\boldsymbol x}^\top\sum^6_{i,j=1}\frac{\delta \mathscr{V}_{ji}(\boldsymbol{\mathscr{V}}^{\top}\boldsymbol{\mathscr{C}}^\top\boldsymbol{d})_j}
{-\rho\tau^2\omega^2+\lambda_j\delta-\mathrm{i}\gamma\sqrt{\rho}\boldsymbol{\mathscr{V}}^\top_j\boldsymbol{\mathscr{R}}
\boldsymbol{\mathscr{V}}_j\omega\delta}\\
\times \int_{\partial D}\boldsymbol{\mathcal{S}}^{-1}_D[\boldsymbol{\xi}_i](1+\mathcal{O}(\delta^{\frac{1}{2\beta_{j}}}))\mathrm{d}\sigma(\boldsymbol y).
\end{align*}
are referred to as the transverse and longitudinal far-field patterns, respectively.
\end{theo}

\begin{proof}
It is clear from \eqref{fun solu far} that
\begin{align*}
\boldsymbol\Gamma^{\sqrt{\rho}\omega}(\boldsymbol x-\boldsymbol y)&=-\frac{e^{\mathrm{i}\frac{\sqrt{\rho}\omega}{\sqrt{\mu}}|\boldsymbol x|}}{4\pi\mu|\boldsymbol x|}(\mathbf{I}-\widehat{\boldsymbol x}\widehat{\boldsymbol x}^\top)-\frac{e^{\mathrm{i}\frac{\sqrt{\rho}\omega}{\sqrt{\lambda+2\mu}}|\boldsymbol x|}}{4\pi(\lambda+2\mu)|\boldsymbol x|}\widehat{\boldsymbol x}\widehat{\boldsymbol x}^\top+\mathcal{O}(\omega|\boldsymbol x|^{-1})+\mathcal{O}(|\boldsymbol x|^{-2})
\end{align*}
as $|\boldsymbol x|\rightarrow +\infty$ and $k\rightarrow0$, where the last two remainder terms are uniformly bounded for $\boldsymbol{y}\in\partial D$. Furthermore, applying \eqref{EIS:series}, \eqref{eq:compressed}, and \eqref{u:asymptotic}, we have that for $\omega\in\mathbb{R}$, with $\omega=\mathcal{O}(\epsilon^{\frac{1}{2}})$,
\begin{align*}
(\boldsymbol{\mathcal{S}}^{\sqrt{\rho}\omega}_{D})^{-1}[\boldsymbol{u}|_{\partial D}-\boldsymbol{u}^{\mathrm{in}}|_{\partial D}]=\sum^6_{i,j=1}\frac{\delta \mathscr{V}_{ji}(\boldsymbol{\mathscr{V}}^{\top}\boldsymbol{\mathscr{C}}^\top\boldsymbol{d})_j(1+\mathcal{O}(\delta^{\frac{1}{2\beta_{j}}}))}
{-\rho\tau^2\omega^2+\lambda_j\delta-\mathrm{i}\gamma\sqrt{\rho}\boldsymbol{\mathscr{V}}^\top_j\boldsymbol{\mathscr{R}}
\boldsymbol{\mathscr{V}}_j\omega\delta}\boldsymbol{\mathcal{S}}^{-1}_{D}[\boldsymbol{\xi}_i]+\mathcal{O}(1).
\end{align*}
By combining this result with the previous equality and \eqref{ext:solu}, we obtain \eqref{far solution}, thereby completing the proof.
\end{proof}

\section{Concluding remarks}\label{sec:conclusion}

This paper presents a mathematical analysis of subwavelength resonances induced by high contrast elastic media. By introducing the DtN map, we reformulate the scattering problem in $\mathbb{R}^{3}$ as a boundary value problem defined on the bounded domain $D$. We then introduce an auxiliary sesquilinear form that is both bounded and coercive, allowing us to establish necessary and sufficient conditions for the well-posedness of the solution to this boundary value problem. Subwavelength resonances are characterized by the condition that the determinant of a specific matrix vanishes. Utilizing this explicit characterization, along with Gohberg--Sigal theory and Puiseux series expansions for
multi-valued functions, we derive asymptotic expansions for the frequencies at which these resonances occur. The leading terms of these expansions are closely related to the eigenvalues of a real-valued symmetric positive definite matrix. Furthermore, we establish the asymptotic expansion of the field in the domain $D$, where the enhancement coefficients are determined by the imaginary components of the resonant frequencies. We also obtain the far-field pattern, which includes both the transversal and longitudinal components.

To the best of our knowledge, this paper is the first to develop a variational method for analyzing subwavelength resonances in elastic systems, avoiding the need for layer potential techniques. As mentioned in the introduction, this method offers significant flexibility, as it does not rely on the spectra of single-layer potential or Neumann--Poincar\'{e} operator. Our study paves the way for new mathematical investigations into subwavelength resonances associated with the Lam\'{e} system and other scattering problems. In future work, we aim to extend this method to explore scattering problems involving time-modulated or quasi-periodic elastic waves.

 \end{document}